\documentclass[11pt]{article}
\usepackage{epsfig}
\usepackage{amssymb,amsmath,amsthm,amscd}
\usepackage{latexsym}

\newtheorem{thm}{Theorem}[section]
\newtheorem{prop}[thm]{Proposition}
\newtheorem{cor}[thm]{Corollary}
\newtheorem{lem}[thm]{Lemma}

\theoremstyle{definition}
\newtheorem{defn}[thm]{Definition}
\newtheorem{rem}[thm]{Remark}

\newcounter{labelflag} 

\newcommand{\Label}[1]{
                       \ifnum\thelabelflag=1
                          \ifmmode
                             \makebox[0in][l]{\qquad\fbox{\rm#1}}
                          \else
                             \marginpar{\vspace{0.7\baselineskip}
                                        \hspace{-1.1\textwidth}
                                        \fbox{\rm#1}}
                          \fi
                       \fi
                       \label{#1} }

\newcommand{\be}{\begin{equation}}
\newcommand{\ee}{\end{equation}}

\newcommand{\R}{\mathbb{R}}
\newcommand{\N}{\mathbb{N}}

 \def \rn { {{\mathbb{R}}^N } }

\def \cq {{  C( {\overline{Q}} ) }}
\def \cqzero {{  C_0 ( {\overline{Q}} ) }}

\def \calf {{  {\mathcal{F}} }}
 \def \calftwo {{  {\mathcal{F}} }}
\def \cala {{  {\mathcal{A}}  }}
 \def \calb {{  {\mathcal{B}}  }}
\def \cald {{  {\mathcal{D}}  }}

\def \vt { {\tilde{v} }}

\def \thtwot {{  \theta_{t}  }}

 \def  \ltwo {{L^2 (Q)}}

\begin{document}

\begin{titlepage}
\title{\Large  \bf    Existence, Stability and Bifurcation of Random 
Complete and Periodic 
Solutions of Stochastic Parabolic Equations }
\vspace{10mm}

\author{
Bixiang Wang 
\vspace{5mm}\\
Department of Mathematics\\
 New Mexico Institute of Mining and
Technology  \\ Socorro,  NM~87801, USA  \vspace{3mm}\\
Email: bwang@nmt.edu }
\date{}
\end{titlepage}

\maketitle

\medskip

\begin{abstract}
In this paper,   we study 
 the existence, stability and bifurcation  of  random 
 complete    and periodic solutions
for 
stochastic parabolic equations with multiplicative noise.
We  first prove  the  existence and uniqueness
of tempered random attractors  for the stochastic equations  and 
characterize the structures  of the  attractors by random complete solutions.
We  then  examine  the existence  and stability of random complete
quasi-solutions and establish the relations of these solutions and
the structures of tempered  attractors.
When the stochastic equations are incorporated 
with periodic  forcing,   we obtain   the existence and stability of random
periodic solutions. 
  For the stochastic Chafee-Infante equation, we further establish 
  the multiplicity  and  stochastic    bifurcation  of   complete    and periodic solutions.
 \end{abstract}



\baselineskip=1.3\baselineskip

\section{Introduction}
\setcounter{equation}{0}

 In this paper,  
 we study the existence, stability and bifurcation of random 
 complete    and periodic 
 solutions of a class of stochastic parabolic equations
 with multiplicative noise. 
 Let   $Q$  be a bounded smooth  open set in $\R^n$ with boundary
$\partial Q$.
Given $\tau \in \R$, 
consider    the  following equation  
     on $ (\tau, \infty) \times Q $: 
 \be
 \label{intro1}
  {\frac {\partial u}{\partial t}}
     - \Delta u     =  f(t,x, u)   +  g(t, x) 
     +  \alpha u \circ  {\frac {d \omega}{dt}},
      \quad x \in Q \ \ \mbox{ and } \ t >\tau, 
 \ee
 with  initial condition 
 $u(\tau, \cdot) = u_\tau$
 and homogeneous Drichlet
 boundary condition, 
where   $\alpha $ is a  positive constant,
$g\in L^1_{loc} (\R, L^\infty (Q))$, 
$\omega$  is a  two-sided real-valued Wiener process on a probability space,
and $f$ is a   locally Lipschitz continuous function satisfying  
some structural conditions.
The stochastic equation \eqref{intro1}
is understood  
in the sense of  Stratonovich  integration.

It is known   that complete solutions play an important role
in  determining the long time dynamics of evolution equations.
For instance, the structures of global attractors of 
some deterministic
systems are completely characterized  by bounded complete solutions,
see, e.g., \cite{bab1, bal2, carva1, carva2,  hal1, sel1, tem1}.
Similarly, the structures of  random attractors
of stochastic equations  are fully
described by   random complete solutions
as demonstrated in \cite{wan4}.
Note that random traveling wave solutions can be considered as
a class of random complete solutions, and such solutions 
have  been introduced  by Shen in \cite{shen1} and studied
in \cite{nole1, nole2, nole3, nole4}.
Since parabolic equations possess the comparison principle,
the dynamical systems generated by these equations are monotone.
The dynamics of monotone  systems has been extensively
investigated in the literature, see, e.g., 
\cite{chu1, chu2, hess1,   hir1, lan1,  pola1,  smi1}.
For such systems, complete solutions can be used
to describe the structures of attractors in more detail.
For deterministic monotone equations,  it is possible to
find a stable  complete solution by which the attractor is bounded
from above.  It is also possible to find  such a  
solution  to bound the attractor from below, see, e.g.,
\cite{lan1, lan2} and the references therein.
For autonomous random systems, the role of complete
solutions should be replaced by stationary solutions,
and in this case, random attractors are bounded
by   extremal stationary   solutions from above and below, respectively,
see, e.g., \cite{arn1, chu1, chu2}.

Note that the articles mentioned above 
 deal with either deterministic systems or
autonomous random systems. As far as the author is aware, there is
no analogous result  available in the literature 
regarding  the existence,   stability and bifurcation of
extremal   random complete
solutions for non-autonomous  stochastic equations.
In this paper,  we will  employ the non-autonomous random dynamical
systems theory to investigate such solutions for 
   parabolic equations with multiplicative noise.
More precisely, we will show  that the 
nonlinear stochastic equation
\eqref{intro1}   has a unique 
tempered random attractor in the
space $\cq$ of  continuous functions
on ${\overline{Q}}$  with supremum norm.
We will also prove  that a linear  equation
associated with \eqref{intro1}   possesses a unique
tempered complete quasi-solution in $\cq$
which attracts all solutions.
Then by the comparison principle, we show   that
the nonlinear equation
\eqref{intro1}
has two tempered complete quasi-solutions
$u^*$  and $u_*$, where $u^*$  is 
    maximal  with respect to the random attractor and
  $u_*$ is   minimal.
  The stability of $u^*$ and $u_*$ is also obtained.
    As we will   see later, 
    the maximal solution $u^*$ is stable from above
  and  the   minimal solution $u_*$ is stable from below
  (see Theorem \ref{max1} for more details). 
  Based  on  this result,  we further study the bifurcation problem
  of random complete  solutions of  a specific parabolic equation,
  i.e., the  
 stochastic   
  Chafee-Infante equation.
    We will prove the
  stochastic Chafee-Infante equation undergoes a
  pitchfork bifurcation when a parameter $\nu$ crosses 
  the first eigenvalue $\lambda_1$
  of the negative Laplacian with homogeneous Dirichlet
  boundary condition.
  Actually,   when 
  $\nu \le \lambda_1$,  the origin is the unique
  random complete  quasi-solution and it  is stable. 
  In this case,   the random attractor of the equation 
  is trivial. 
  When $\nu$ passes $\lambda_1$ 
  from below,  the origin loses its stability
  and  two more random complete  quasi-solutions appear.
  This means  
  the equation has at least three random complete quasi-solutions:
  $u^*$, $u_*$ and  $0$ for $\nu > \lambda_1$. 
  Furthermore,   the 
     nontrivial   solutions
  $u^*$   and $u_*$    approach    zero    in $\cq$ when
  $\nu \to \lambda_1$
  (see Theorem \ref{thmbi}). 
  Particularly,   if  $f$   and $g$ in \eqref{intro1}
  are periodic in time, then 
    $u^*$   and $u_*$  become pathwise  random periodic solutions
 (see Theorem \ref{max2}).  
 In this case, we  obtain the bifurcation
 of pathwise random periodic solutions
 that  seems to be the first result  of its kind.
 The reader is referred to \cite{fen1, fen2, zhao1}
 for  details regarding random periodic solutions.
    
  As mentioned before,   the  idea of the present paper  is based on 
  the attractors theory  of
  random dynamical
  systems generated by  non-autonomous stochastic PDEs.
  The existence of random attractors for such systems has been
  established in 
  \cite{car6, dua3, wan4}.
  If  a stochastic equation does not contain deterministic 
  non-autonomous forcing,  then  we say the equation is
  a autonomous stochastic one.
    The concept of 
    random attractor for autonomous systems
  was introduced
   in \cite{cra1, fla1, schm1}, 
   and the existence of  such  attractors
   have been established for various  equations,
   see, e.g.,   \cite{arn1, bat1, bat2, car1, car2, 
   car3, car4, car5, chu1, chu2, cra1, cra2, 
    fla1, gar1, gar2, huang1, kloe1, schm1, wan2, wan3}. 
    For the existence of invariant manifolds 
    for stochastic PDEs,  we refer the reader 
    to \cite{dua1, dua2, lia1, moh1} for details.

   This  paper is organized as follows.
   In the next section, 
   we recall  some  basic  concepts  regarding 
   random attractors of non-autonomous stochastic equations.
   In Section 3,  we    define a continuous random dynamical
   system for  the stochastic equation \eqref{intro1}.
   Section 4 is devoted to 
    the existence of tempered random attractors
   and   periodic attractors for  \eqref{intro1}.
   In Section 5, we discuss  the 
    structures of the attractors  
    as well as  the  
    existence and stability
   of random  complete  quasi-solutions and  random periodic solutions.
   We  then  investigate the bifurcation of random complete and  periodic solutions
   of the Chafee-Infante equation in the last section.

Throughout  this paper,    we  will use  
$\cq$ to denote the space of continuous functions on ${\overline{Q}}$
with supremum norm. We will also use 
$\cqzero$ to denote 
 the     subspace of $\cq$
 which consists   of continuous functions vanishing 
    at $\partial Q$.  The norm of a       Banach space $X$  is written as    $\|\cdot\|_{X}$.
   The letters $c$ and $c_i$ ($i=1, 2, \ldots$)
are  generic positive constants  which may change their  values occasionally.

\section{Preliminaries}
\setcounter{equation}{0}

 For the reader\rq{}s convenience, 
    we recall some concepts from \cite{wan4} regarding 
pullback   attractors    of  non-autonomous  stochastic equations.
The reader is also referred  to  
  \cite{bat1, cra1, cra2, fla1, schm1} for  similar results
  on autonomous stochastic equations, 
   and to \cite{bab1, bal2, hal1, sel1, tem1}
for deterministic attractors.   

Hereafter,  we assume  that
  $(X, d)$   is  a complete
separable  metric space
and 
 $(\Omega, \calftwo, P,  \{\thtwot\}_{t \in \R})$
 is a metric dynamical system.
 Let $\cald$   be 
 a  collection  of  some families of  nonempty subsets of $X$
 parametrized by $\tau \in \R$ 
 and $\omega \in \Omega$.
 A mapping 
$K: \R \times \Omega  \to  2^X$   with closed  nonempty
images is said to be measurable if $K(t, \cdot)$ is
  $(  \calftwo, \ \calb(X) )$-measurable
for every  fixed $t \in \R$.

\begin{defn} \label{ds1}
A mapping $\Phi$: $ \R^+ \times \R \times \Omega \times X
\to X$ is called a continuous  cocycle on $X$
over 
$(\Omega, \calftwo, P,  \{\thtwot\}_{t \in \R})$
if   for all
  $\tau\in \R$,
  $\omega \in   \Omega $
  and    $t, s \in \R^+$,  the following conditions (i)-(iv)  are satisfied:
\begin{itemize}
\item [(i)]   $\Phi (\cdot, \tau, \cdot, \cdot): \R ^+ \times \Omega \times X
\to X$ is
 $(\calb (\R^+)   \times \calftwo \times \calb (X), \
\calb(X))$-measurable;

\item[(ii)]    $\Phi(0, \tau, \omega, \cdot) $ is the identity on $X$;

\item[(iii)]    $\Phi(t+ s, \tau, \omega, \cdot) =
 \Phi(t,  \tau +s,  \theta_{s} \omega, \cdot) 
 \circ \Phi(s, \tau, \omega, \cdot)$;

\item[(iv)]    $\Phi(t, \tau, \omega,  \cdot): X \to  X$ is continuous.
    \end{itemize}
    
    If,  in addition,  there exists  a
    positive number   $T $ such that
    $
\Phi(t,  \tau +T, \omega, \cdot)
= \Phi(t, \tau,  \omega, \cdot )$
    for all  $t  \in \R^+$,
     $\tau \in \R$  and $\omega \in \Omega$,
 then $\Phi$ is called  
a  continuous periodic  cocycle  on $X$ with period $T$.
\end{defn}

\begin{defn}
\label{comporbit}
      (i)   A mapping $\psi: \R \times \R \times \Omega$
 $\to X$ is called a complete orbit (solution) of $\Phi$ if for every  $t \in \R^+ $,
 $s, \tau \in \R$ and $\omega \in \Omega$,   
$ 
 \Phi (t,  \tau +s, \theta_{s} \omega,
  \psi (s, \tau, \omega) )
  = \psi (t +  s, \tau, \omega )
$.
 If, in  addition,    there exists $D=\{D(\tau, \omega): \tau \in \R,
 \omega \in \Omega \}\in \cald$ such that
 $\psi(t, \tau, \omega)$ belongs to
 $D ( \tau +t, \theta_{ t} \omega )$
 for every  $t \in \R$, $\tau \in \R$
 and $\omega \in \Omega$, then $\psi$ is called a
 $\cald$-complete orbit (solution)  of $\Phi$.

      (ii)   A mapping $\xi:  \R   \times \Omega$
 $\to X$ is called a complete quasi-solution  of $\Phi$ if for every  $t \in \R^+ $,
 $\tau \in \R$ and $\omega \in \Omega$,  
$ 
 \Phi (t,  \tau,  \omega,
  \xi (\tau, \omega) )
  = \xi (  \tau +t,  \theta_t \omega )
$.
 If, in  addition,    there exists $D=\{D(\tau, \omega): \tau \in \R,
 \omega \in \Omega \}\in \cald$ such that
 $\xi (\tau, \omega)$ belongs to
 $D ( \tau,  \omega )$
 for all   $\tau \in \R$
 and $\omega \in \Omega$, then $\xi$ is called a
 $\cald$-complete   quasi-solution  of $\Phi$. 
 
 (iii) 
A complete  quasi-solution $\xi$ of $\Phi$  is said to be periodic with
period $T$ if  $\xi (\tau +T, \omega) =\xi(\tau, \omega)$
for all $\tau \in \R$  and $\omega \in \Omega$.
Such a solution is   called   
  a random periodic  solution   in \cite{zhao1}.
  \end{defn}

\begin{defn}
\label{asycomp}
A cocycle 
$\Phi$ is said to be  $\cald$-pullback asymptotically
compact in $X$ if
for all $\tau \in \R$ and
$\omega \in \Omega$,    the sequence
$ 
\{\Phi(t_n,    \tau -t_n, \theta_{ -t_n} \omega,
x_n)\}_{n=1}^\infty$   has a convergent  subsequence  in     
$X $ 
 whenever
  $t_n \to \infty$, and $ x_n\in   B(\tau -t_n,
  \theta_{ -t_n} \omega )$   with
$\{B(\tau, \omega): \tau \in \R, \ \omega \in \Omega
\}   \in \mathcal{D}$.
\end{defn}

\begin{defn}
\label{defatt}
 Let $\cald$ be a collection of some families of
 nonempty  subsets of $X$
 and
 $\cala = \{\cala (\tau, \omega): \tau \in \R,
  \omega \in \Omega \} \in \cald $.
Then     $\cala$
is called a    $\cald$-pullback    attractor  for
  $\Phi$
if the following  conditions (i)-(iii) are  fulfilled:
\begin{itemize}
\item [(i)]   $\cala$ is measurable
  and
 $\cala(\tau, \omega)$ is compact for all $\tau \in \R$
and    $\omega \in \Omega$.

\item[(ii)]   $\cala$  is invariant, that is,
for every $\tau \in \R$ and
 $\omega \in \Omega$,
$$ \Phi(t, \tau, \omega, \cala(\tau, \omega)   )
= \cala ( \tau +t, \theta_{t} \omega
), \ \  \forall \   t \ge 0.
$$

\item[(iii)]   For every
 $B = \{B(\tau, \omega): \tau \in \R, \omega \in \Omega\}
 \in \cald$ and for every $\tau \in \R$ and
 $\omega \in \Omega$,
$$ \lim_{t \to  \infty} d (\Phi(t, \tau -t,
 \theta_{-t}\omega, B(\tau -t, 
 \theta_{-t}\omega) ) , \cala (\tau, \omega ))=0.
$$
 \end{itemize}
 If, in addition, there exists $T>0$ such that
 $$
 \cala(\tau +T, \omega) = \cala(\tau,    \omega ),
 \quad \forall \  \tau \in \R, \forall \
  \omega \in \Omega,
 $$
 then we say $\cala$ is periodic with period $T$.
\end{defn}
 
  The following result  can be found in  \cite{wan4}. 
For similar results, see 
\cite{bat1, car6,  cra2, fla1, schm1}.

\begin{prop}
\label{att}  
 Let $\cald$ be  an inclusion-closed
 collection of some  families of   nonempty subsets of
$X$,  and $\Phi$  be a continuous   cocycle on $X$
over  
$(\Omega, \calftwo, P,  \{\thtwot\}_{t \in \R})$.
If   
$\Phi$ is $\cald$-pullback asymptotically
compact in $X$ and $\Phi$ has a  
closed  
   measurable  
     $\cald$-pullback absorbing set
  $K$ in $\cald$,  then
$\Phi$ has a  $\cald$-pullback
attractor $\cala$  in $\cald$.
   The $\cald$-pullback
attractor $\cala$   is unique   and is given  by,
for each $\tau  \in \R$   and
$\omega \in \Omega$,
$$
\cala (\tau, \omega)
=\Omega(K, \tau, \omega)
=\bigcup_{B \in \cald} \Omega(B, \tau, \omega)
$$
$$
 =\{\psi(0, \tau, \omega): \psi  \text{ is a   }  \cald  \text{-}
 \text{complete  solution  of } \Phi \} 
$$
$$
 =\{\xi( \tau, \omega): \xi  \text{ is a   }  \cald  \text{-}
 \text{complete   quasi-solution  of } \Phi \} .
$$

If, in addition, both $\Phi$  and $K$ are $T$-periodic, then so is the attractor
$\cala$.
  \end{prop}

\begin{rem}
We emphasize that   
the  attractor   $\cala$
in Propositions \ref{att}  
is $(\calf, \calb (X))-$measurable 
which was proved   in
\cite{wan8}. While,    the  measurability of $\cala$
was only proved  in \cite{wan4}  with respect to 
   the $P$-completion
 of $\calf$.
\end{rem}

\section{Nonlinear Stochastic  Equations}
\setcounter{equation}{0}

In this section,  we  introduce the nonlinear stochastic PDEs
 we will study. 
Suppose   $Q$  is a bounded smooth  open set in $\R^n$ with boundary
$\partial Q$.
Consider       the   stochastic  parabolic 
   equations  with multiplicative noise 
    defined on $ (\tau, \infty) \times Q  $
   with  $\tau \in \R$:  
 \be
 \label{rde1}
  {\frac {\partial u}{\partial t}}
     - \Delta u     =  f(t,x, u)   +  g(t, x) 
     +  \alpha u \circ  {\frac {d \omega}{dt}},
      \quad x \in Q \ \ \mbox{ and } \ t >\tau, 
 \ee
 with boundary condition
 \be\label{rde2}
 u = 0,   \quad x \in  \partial Q \ \ \mbox{ and } \ t >\tau, 
 \ee
 and   initial condition
 \be\label{rde3}
 u(\tau,  x) = u_\tau (x),   \quad x\in  Q,
 \ee
where   $\alpha $ is a  positive constant,
$g\in L^1_{loc} (\R, L^\infty (Q))$, 
$\omega$  is a  two-sided real-valued Wiener process on a probability space,
and the symbol  $\circ$   indicates that the equation is understood
in the sense of  Stratonovich  integration.
 The nonlinearity    $f: \R \times {\overline{Q}} \times \R
 \to \R$ is continuous in 
 $(t,x,s) \in \R \times {\overline{Q}} \times \R$.
 We further assume that $f$
   is    locally  Lipschitz  continuous
  in $s \in \R$  in the sense  that for any bounded intervals
  $I_1$  and $I_2$,  there exists   $L>0$ (depending on $I_1$ and $I_2$)
  such that for all $t \in I_1$, $s_1, s_2 \in I_2$ and $x \in Q$, 
  \be
  \label{f1}
  |f(t, x, s_1) - f(t, x, s_2)|
  \le L  |s_1 -s_2| .
  \ee
 
 Let $\lambda$ be the first eigenvalue of the negative Laplacian on $Q$
 with homogeneous Dirichlet boundary condition.
 Suppose that  there exist   $\beta \in (0, \lambda)$
 and $ h \in L^1_{loc}(\R, L^\infty (Q))$
  such that
 \be
 \label{f2}
 f(t,x, 0) =0
 \ \mbox{ and } \ 
 f(t,x, s) s \le 
  \beta s^2 + h(t,x) |s|, \quad \mbox{   for all } \  t \in \R, \ x \in Q
  \ \mbox{ and }  \ s \in \R.
 \ee
Throughout this section, we 
 assume that $\delta$ 
is a  fixed constant
such that
\be\label{delta1}
0 < \delta < \lambda -\beta.
\ee
 Suppose  $g$  and $h$ satisfy the following condition: 
 for every $\tau \in \R$, 
 \be
\label{gh1}
\int_{-\infty}^\tau   e^{\delta s}
\left (
 \| g (s, \cdot) \|_{L^\infty(Q)}
 +
 \| h (s, \cdot) \|_{L^\infty(Q)}
 \right ) ds
< \infty,
\ee
where $\delta$ is as in  \eqref{delta1}.
Sometimes,  
 we also assume  $g$ and $h$ 
have the property:
for every $c>0$,  $\tau \in \R$   and
$\omega \in \Omega$,
\be\label{gh2}
\lim_{r \to -\infty}
e^{  c r } \int_{-\infty}^{0}
  e^{\delta s}
 \left (
 \| g (s+r, \cdot)\|_{L^\infty(Q)}
 +
 \| h (s+r, \cdot)\|_{L^\infty(Q)}
 \right )
  ds =0.
 \ee
 Note that \eqref{gh2} implies \eqref{gh1}.
 To describe   the   probability space  that   will be   used
in this paper,  we write  
$ 
\Omega = \{ \omega   \in C(\R, \R ): \ \omega(0) =  0 \}
$. 
Let $\calf$  be
 the Borel $\sigma$-algebra induced by the
compact-open topology of $\Omega$, and $P$
be  the corresponding Wiener
measure on $(\Omega, \calf)$.   
 There is a classical group  $\{\thtwot \}_{t \in \R}$  acting on  
 $(\Omega, \calf, P)$
    given by
$
 \thtwot \omega (\cdot) = \omega (\cdot +t) - \omega (t)$
for all  $ \omega \in \Omega$ and $ t \in \R$.
In addition,        
  there exists a $\thtwot$-invariant set 
  $\tilde{\Omega}\subseteq \Omega$
of full $P$ measure  such that
for each $\omega \in \tilde{\Omega}$, 
\be\label{aspomega}
  {\omega (t)}/t \to 0 \quad \mbox {as } \ t \to \pm \infty.
\ee
Hereafter,  we only consider the space 
$\tilde{\Omega}$, and hence
    write  $\tilde{\Omega}$ as 
$\Omega$  for   convenience. 
 Let $u(t, \tau, \omega, f, g, u_\tau)$ be  the solution
 of problem \eqref{rde1}-\eqref{rde3} and 
  $v$ be a new variable given by
 \be\label{vu}
 v(t, \tau, \omega,f, g,  v_\tau) = z(t, \omega)
 u(t, \tau, \omega, f, g, u_\tau)
 \quad \mbox{ with } \ v_\tau = z(\tau, \omega) u_\tau,
\ee
where   $ z(t, \omega) = e^{-\alpha \omega  (t)}$. 
From \eqref{rde1}-\eqref{rde3}  
 we get 
 \be
 \label{v1}
  {\frac {\partial v}{\partial t}}
     - \Delta v     = z(t, \omega)  f(t,x,  z^{-1} (t, \omega)  v )    +  z(t, \omega)  g(t, x) ,
      \quad x \in Q \ \ \mbox{ and } \ t >\tau, 
 \ee
 with  
 \be\label{v3}
 v|_{\partial Q} = 0   \ \mbox{ and }  
 v(\tau, x) =v_\tau (x).
 \ee
 It follows from \cite{mora1}   that
for every $\tau \in \R$, $\omega \in \Omega$
and $v_\tau  \in \cqzero$,  there exists $T>0$
such that   the deterministic 
problem \eqref{v1}-\eqref{v3} has a unique solution
$v \in C([\tau, \tau +T), \cqzero)$ given by the
variation of constants formula
\be
\label{nonv}
v(t, \tau, \omega, f, g, v_\tau)
=e^{\Delta (t-\tau)} v_\tau
+\int_\tau^t z(s, \omega)
e^{\Delta (t-s) }
\left (
 f(s, \cdot , z^{-1} (s,\omega) v(s)) + g(s, \cdot)
\right ) ds,
\ee
for all $t \in [\tau, \tau +T)$, where $\Delta$ is the Laplacian.
Moreover,  the solution depends continuously
on $v_\tau$ in $\cqzero$ and is measurable with respect
to $\omega \in \Omega$.  
By Lemma \ref{lgsol}  below, the 
    the solution
$v$ is actually
defined   for all   $t >\tau$.
Note that condition   \eqref{f2}  implies
 for every $\omega \in \Omega$,
 \be\label{f2_2}
 z(t,\omega) f(t,x, z^{-1} (t, \omega) s) s \le \beta s^2 + z(t, \omega) h(t,x) |s|,
  \quad   \  t \in \R, \ x \in Q
  \ \mbox{ and }  \ s \in \R.
  \ee
 We will use  \eqref{f2_2}   
 to control the solutions of
  \eqref{v1}
 by   the    linear  equation:
 \be
 \label{lv1}
  {\frac {\partial \vt}{\partial t}}
     - \Delta \vt  -\beta \vt   = z(t, \omega) 
     \left ( h(t,x) +   |g(t, x)|
     \right ) ,
      \quad x \in Q \ \ \mbox{ and } \ t >\tau, 
 \ee
 with 
 \be\label{lv3}
 \vt |_{\partial Q} = 0,     \ \ \mbox{ and }   \
 \vt (\tau, x) =\vt_\tau (x),  \
   x\in  Q.
 \ee
 For convenience,    we write  $A = -\Delta -\beta I$.
Then it is known  that 
$A$ is  a generator of an   analytic semigroup,
denoted  $\{e^{-At}\}_{t \ge 0}$, in $\cqzero$
  (see, e.g., \cite{mora1}).
 Given $\tau \in \R$, $\omega \in \Omega$
and $\vt_\tau  \in \cqzero$,  problem \eqref{lv1}-\eqref{lv3}
has a  solution
$\vt \in ([\tau, \infty), \cqzero) $  given by
 \be\label{linv}
\vt(t, \tau, \omega, g, h, \vt_\tau)
=e^{-A(t-\tau)}\vt_\tau
+\int_\tau^t  z(s, \omega) e^{-A(t-s)} 
\left (h(s, \cdot)  + |g(s, \cdot)| 
\right ) ds,
\ee
for all $t >\tau$.
A  mapping 
 $\xi: \R \times \Omega \to \cqzero$    is called a 
  complete quasi-solution of problem
\eqref{lv1}-\eqref{lv3} if 
for every  $\tau \in \R$,  $t >0$ and $\omega \in \Omega$,
  \be\label{colineq}
   \vt(t+\tau, \tau, \omega^{-\tau},g, h,  \xi(\tau, \omega))
  =\xi (\tau +t, \omega^t),
  \ee
  where $\omega^t$ is the translation
 of $\omega$ by $t$;     that is,
 $\omega^t(\cdot) = \omega (\cdot +t)$. 
 Note that such a mapping $\xi$
 is used in  \cite{zhao1} for the definition
 of random periodic solutions.
 Similarly,   from now on, 
 we will use  $f^t(\cdot, x, s)$,
 $g^t(\cdot,x)$  and
 $h^t(\cdot, x)$  for  the
 translations of $f$, $g$ and $h$
 in their first argument by $t$, respectively. 
   The dynamics of  the linear equation \eqref{lv1} is well understood
   as presented below.

\begin{lem}
\label{lcoms2}
  Suppose  \eqref{gh1}-\eqref{gh2}  hold.
Then  problem  \eqref{lv1}-\eqref{lv3}
has a unique tempered  complete  quasi-solution  $\xi:
\R  \times \Omega \to \cqzero$, which is given by,
for each $t \in \R$ and $\omega \in  \Omega$,
 \be\label{lcoms2a1}
 \xi  (t, \omega ) =
  \int_{-\infty} ^0  e^{ A s } z(s, \omega)  
  \left ( h^t (s, \cdot)  + |g^t (s, \cdot)| \right) ds.
 \ee
 In addition,  $\xi$ pullback attracts all solutions
 in the sense   that
 for every $t >0$, $\tau \in \R$ and $\omega \in  \Omega$,
 $$
\| \vt (\tau, \tau-t,  \omega^{-\tau}, g,h,  \vt_{\tau-t})
-\xi (\tau, \omega)\|_\cqzero
$$
\be\label{lcoms2a2}
\le
M e^{-(\lambda - \beta) t}
\left (
\|\vt_{\tau -t}\|_\cqzero
+\|  \xi (\tau-t, \omega^{-t})  \|_\cqzero 
\right ),
\ee
 where  $M $ is a positive constant
 independent of $t$, $\tau$ and $\omega$.

 Furthermore,  if    $g$  and   $ h$
   are  periodic with
 period $T>0$,  then  so is $\xi$, 
 i.e.,  $\xi (t+T, \omega) = \xi(t, \omega)$ for
 all $t \in \R$  and $\omega \in  \Omega$.
\end{lem}

\begin{proof}
The proof is quite standard, see, e.g., 
\cite{fen1, fen2}.
First, by \eqref{delta1} and  \eqref{gh1}, one can verify
that the integral 
on the right-hand side of  \eqref{lcoms2a1}
is well defined.
 We now show   that $\xi$ is a complete quasi-solution.
 For convenience, we write
 $\varphi (t,x) = h(t,x) + |g(t,x)|$.
 By \eqref{linv} and
  \eqref{lcoms2a1}  we have
  $$ \vt(t, 0, \omega, g^\tau,  h^\tau,   \xi(\tau, \omega))
  =e^{-At} \xi (\tau, \omega)
  + \int_0^t e^{-A (t-s)} z(s, \omega) \varphi^\tau (s, \cdot) ds
  $$
   \be\label{coms1_p21}
  =  \int_{-\infty}^t e^{-A (t-s)} z(s, \omega) \varphi^\tau (s, \cdot) ds
  =  \int_{-\infty}^0 e^{A  s } z(s, \omega^t)
  \varphi^{t +\tau} (s , \cdot) ds.
 \ee
 It follows   from \eqref{lcoms2a1} and \eqref{coms1_p21}  that
 for each $\tau \in \R$,  $t >0$ and $\omega \in  \Omega$,
 $$
   v(t, 0, \omega,\varphi^\tau,  \xi(\tau, \omega))
  =\xi (\tau +t, \omega^t),
  $$
  which   implies    \eqref{colineq} and hence 
   $\xi$  is a complete quasi-solution of problem
  \eqref{lv1}-\eqref{lv3}.
  Next we prove that $\xi$ is tempered.
  Given $c>0$,  let $\nu = {\frac 12} 
  \min \{ c, \  (\lambda -\beta -\delta) \}$.
  Note  that   for each  $\omega \in \Omega$,   there exists
 $s_0 <0$ such that for all  $s \le s_0$,
 \be
 \label{coms1_p23}
 -\alpha \omega (s) \le - \nu  s.
 \ee
 By \eqref{coms1_p23} we have
 for  all $t \le s_0$, $\tau \in \R$ and $\omega \in  
 \Sigma$, 
 $$
 e^{ct} \| \xi(\tau +t, \omega^t) \|_{\cqzero}
 \le e^{ct} \int_{-\infty}^0
 \| e^{As} z(s, \omega^t)
  \varphi^{\tau +t} (s, \cdot) \|_{L^\infty(Q)} ds
 $$
 $$
 \le c_1 e^{ct} \int_{-\infty}^0
 e^{(\lambda -\beta) s} e^{-\alpha \omega (s+t) }
 \| \varphi (s+\tau +t)\|_{L^\infty(Q)} ds
 $$
 $$
 \le c_1 e^{(c-\nu) t} \int_{-\infty}^{0}
 e^{{\frac 12} (\lambda -\beta -\delta) s} e^{\delta s}
 \| \varphi (s+\tau +t)\|_{L^\infty(Q)} ds
 \le c_1 e^{{\frac 12}  c  t} \int_{-\infty}^{0}
  e^{\delta s}
 \| \varphi (s+\tau +t)\|_{L^\infty(Q)} ds,
 $$
 from which   and \eqref{gh2},  we get that 
 for every $c>0$, $\tau \in \R $   and $\omega 
 \in \Omega$, 
  $$
 \lim_{t  \to -\infty}
 e^{ct} \| \xi(\tau +t, \omega^t) \|_{\cqzero}
 =0.
 $$
 Therefore $\xi$ is   tempered.
 We now establish  the stability of  $\xi$.
It follows   from \eqref{linv} 
and \eqref{lcoms2a1}  that,
for every $t >0$, $\tau \in \R$
and $\omega \in  \Omega$,
$$
\vt(\tau, \tau-t, \omega^{-\tau},g,h,  \vt_{\tau-t})
-\xi (\tau, \omega)
=
\vt(t,  0, \omega^{- t},\varphi^{\tau -t},  \vt_{\tau-t})
-\xi (\tau, \omega)
$$
$$
=
e^{-At} \vt_{\tau -t}
+ \int_0^t e^{-A(t-s)} z(s, \omega^{-t}) \varphi^{\tau -t}
(s, \cdot) ds -\xi (\tau, \omega)
$$
\be\label{coms1_p40}
=
e^{-At} \vt_{\tau -t}
- \int_{-\infty} ^{-t}  e^{ A s } z(s, \omega) \varphi^\tau
(s , \cdot) ds .
\ee
On the other hand, by
\eqref{lcoms2a1} we have
$$
e^{-At} 
\xi (\tau-t, \omega^{-t})
=\int_{-\infty}^{0}
e^{A(s-t)} z(s, \omega^{-t}) \varphi^{\tau -t} (s, \cdot) ds
$$
\be
\label{coms1_p41}
=\int_{-\infty}^{-t}
e^{As } z(s+t, \omega^{-t}) \varphi^{\tau -t} (s +t, \cdot) ds
=
\int_{-\infty}^{-t}
e^{As } z(s, \omega) \varphi^{\tau } (s , \cdot) ds.
\ee
By \eqref{coms1_p40}-\eqref{coms1_p41}   we get that, 
for every $t >0$, $\tau \in \R$
and $\omega \in \Omega$,
$$
\|\vt(\tau, \tau-t,  \omega^{-\tau},  g, h, \vt_{\tau-t})
-\xi (\tau, \omega)\|_\cqzero
=
\| e^{-At } (\vt_{\tau-t} - \xi (\tau-t, \omega^{-t}) ) \|_\cqzero
$$
\be\label{coms1_p42} 
\le c_1 e^{-(\lambda - \beta) t}
\left (
 \|\vt_{\tau -t}\|_\cqzero   +  \|  \xi (\tau-t, \omega^{-t})  \|_\cqzero 
 \right ).
	\ee
	Note that the uniqueness of tempered complete quasi-solutions
	of \eqref{lv1}-\eqref{lv3} is implied by
	\eqref{coms1_p42}.  This completes   the proof.
\end{proof}

Next, we establish the global existence of solutions for
the deterministic problem \eqref{v1}-\eqref{v3}.

\begin{lem}
\label{lgsol}
 Suppose   \eqref{f1}-\eqref{f2}   hold. 
 Then for every $\tau  \in \R$, $\omega \in \Omega$ and
  $v_\tau  \in \cqzero$,  the solution 
   $v(t, \tau, \omega, f, g,  v_\tau )$
  of
  problem  \eqref{v1}-\eqref{v3} 
  is defined for all $t \ge \tau$.
  \end{lem}

 \begin{proof}
 Suppose $[\tau, T)$ be the maximal interval of
 existence of the solution  
 $v(t, \tau, \omega, f, g,  v_\tau )$ with
 $T<\infty$. We only need to prove that   
   $v(t, \tau, \omega, f, g,  v_\tau )$
    is bounded in $\cqzero$
 for all $t \in [t_0, T)$.  
 Let  $\vt (t, \tau, \omega,  g, h,   |v_\tau| )$
 be the solution of the linear problem
 \eqref{lv1}-\eqref{lv3} with initial condition
 $|v_\tau|$.   
 By the comparison principle we get, 
  	 for all $t \ge \tau $, 
  	\be
  \label{lgsol_p3}
  |v(t, \tau, \omega, f, g,  v_\tau )|
  \le 
  \vt (t, \tau, \omega,  g, h,   |v_\tau| ).
  \ee  
 By \eqref{linv}   we obtain,  for all $t \in [\tau, T)$,
 $$
 \| \vt (t, \tau, \omega,  g, h,   |v_\tau| ) \|_{\cqzero}
\le  \| e^{-A(t-\tau)} |v_\tau| \|_\cqzero
+ \| \int_\tau^t  z(s, \omega) e^{-A(t-s)} 
\left (h(s, \cdot)  + |g(s, \cdot)| 
\right ) ds \|_\cqzero 
$$
$$
\le c e^{-(\lambda -\beta) (t- \tau)}
 \|  v_\tau \|_\cqzero
+ c\int_{\tau}^t  |z(s, \omega)| e^{-(\lambda -\beta) (t-s)} 
  \left ( \|g(s, \cdot)\|_{L^\infty(Q)}
  + \|h(s, \cdot)  \|_{L^\infty(Q)} \right ) ds .
  $$
  Since $g,\ h \in L^1_{loc} (\R, L^\infty(Q))$,   we find
  from the above that there 
   exists  $c_1>0$ such that
    $$
    \| \vt (t, \tau, \omega,  g, h,   |v_\tau| ) \|_{\cqzero}
\le c_1, \quad \mbox{ for all } \  \tau \le t <T,
$$ 
which along with 
 \eqref{lgsol_p3}  concludes
  the proof.   
  \end{proof}
  
  By \eqref{vu}   and Lemma \ref{lgsol},
   we can define  
   a map
       $\Phi: \R^+ \times \R \times \Omega \times \cqzero$
$\to \cqzero$ for      problem
\eqref{rde1}-\eqref{rde3}.
Given $t \in \R^+$,  $\tau \in \R$, $\omega \in \Omega$ 
 and $u_\tau \in \cqzero$,
let 
 \be \label{rdephi}
 \Phi (t, \tau,  \omega, u_\tau) = 
  u (t+\tau,  \tau, \theta_{ -\tau} \omega, 
  f, g, u_\tau) 
  = {\frac 1{z(t+\tau, \theta_{ -\tau} \omega)}}
v(t+\tau, \tau,  \theta_{ -\tau} \omega, g, h,  v_\tau),  
\ee
where $v_\tau = z(\tau, \theta_{ -\tau} \omega) u_\tau $. 
Note that  $v$  is  continuous in $v_\tau$
in $\cqzero$ and  is measurable 
in $\omega \in \Omega$.   One can check   that  $\Phi$ 
is a continuous   cocycle on $\cqzero$ 
over  
$(\Omega, \calf,
P, \{\thtwot \}_{t\in \R})$ in the sense of Definition \ref{ds1}.
By \eqref{vu},  we have  
the following identities which are useful in later sections, 
    for   each $\tau \in \R$, 
 $\omega \in \Omega$ and $t \ge 0$,    
  $$
   u (\tau, \tau -t,  \theta_{ -\tau} \omega, 
f, g, u_{\tau -t}  )
=
 u (0,  -t,   \omega, 
f^\tau, g^\tau, u_{\tau -t}  )   
$$
\be\label{vuid}
=  v (0,  -t,   \omega, 
f^\tau, g^\tau,   z(-t,  \omega) u_{\tau -t}  ) 
  = v (\tau, \tau -t,   \omega^{-\tau}, 
f, g,    z(-t,  \omega) u_{\tau -t}  ) .
  \ee

    Let  $D =\{ D(\tau, \omega): \tau \in \R, \omega \in \Omega \}$  
     be   a  tempered family of
  bounded nonempty   subsets of $\cqzero $,  that is,  
  for every  $c>0$, $\tau \in \R$   and $\omega \in \Omega$, 
 \be
 \label{temd1}
 \lim_{r \to  - \infty} e^{   c  r} 
 \| D( \tau  +r, \theta_{ r} \omega ) \|_\cqzero  =0, 
 \ee 
 where  we have used the notation
  $\| B \|_\cqzero   =   \sup\limits_{u\in B }
   \|u  \|_{\cqzero }$ for a subset $B$ of $\cqzero$. 
From now on,  we  use  $\cald$      to denote
the   collection of all  tempered families of
bounded nonempty  subsets of $\cqzero$, i.e.,
 \be
 \label{temd2}
\cald = \{ 
   D =\{ D(\tau, \omega): \tau \in \R, \omega \in \Omega \}: \ 
 D  \ \mbox{satisfies} \  \eqref{temd1} \} .
\ee
From \eqref{temd2}   we see  that
$\cald$ is  neighborhood closed.

\section{  Tempered Attractors and Periodic Attractors}
\setcounter{equation}{0}

   In this section, we 
   prove the existence of a unique tempered random 
   attractor for  the stochastic problem
   \eqref{rde1}-\eqref{rde3}
   with non-autonomous term $g$.
   In the case where $f$ and $g$ are 
      periodic,  we show   
   the   attractor is also periodic.
    We first  derive 
     uniform 
   estimates of solutions  in $\cqzero$.

  \begin{lem}
\label{est1}
Suppose   \eqref{f1}-\eqref{gh2}   hold. 
Then for every $\tau \in \R$, $\omega \in \Omega$  
 and $D=\{D(\tau, \omega)
: \tau \in \R,  \omega \in \Omega\}  \in \cald$,
 there exists  $T=T(\tau, \omega,  D) \ge 1$  
 such that for all $t \ge T$
 and $r \in [\tau-1, \tau]$, the solution
 $v$ of  problem  \eqref{v1}-\eqref{v3}    satisfies
 $$
\| v (r, \tau -t,   \omega^{-\tau}, 
f, g, v_{\tau -t}  ) \|_\cqzero
 $$
\be\label{est1_a1}
\le M
 +
    M  e^{- (\lambda -\beta) \tau} \int_{-\infty}^\tau
    e^{(\lambda -\beta) s}  z(s,  \omega^{-\tau})
    \left (\| h(s,\cdot)\|_{L^\infty (Q)} 
    +
    \| g(s,\cdot)\|_{L^\infty (Q)}  \right ) ds,
   \ee
   where $ v_{\tau -t} = z(-t, \omega) u_{\tau -t}$
    with $u_{\tau -t} \in  D(\tau -t, \theta_{ -t} \omega )$, 
     and 
  $M$ is a  positive constant 
  depending on $\lambda$ and $\beta$, but
   independent of $\tau$, $\omega$   and $D$.
\end{lem}

 \begin{proof}
Given $\tau \in \R$, 
$r \in [\tau -1, \tau]$
and    $ t \ge 1$,  by \eqref{linv} we have
 $$ \| \vt(r, \tau -t, \omega^{-\tau}, g, h, v_{\tau -t} )  \|_\cqzero
\le  \|  e^{-A(r-\tau +t)}v_{\tau-t} \|_\cqzero
$$
$$
+\int_{\tau-t} ^r  z(s, \omega^{-\tau}) \| e^{-A(r-s)} 
\left (h(s, \cdot)  + |g(s, \cdot)| 
\right ) \|_{L^\infty} ds
$$
$$
\le c_1 e^{ -(\lambda -\beta) (r-\tau +t)} \| v_{\tau -t}\|_\cqzero
$$
\be\label{lcoms3_p1}
+ c_1 
e^{ -(\lambda -\beta) r}
\int^r_{\tau -t} e^{ (\lambda -\beta) s} z(s, \omega^{-\tau})
\left (
\| h(s, \cdot)\|_{L^\infty (Q)}
+ \| g(s, \cdot) \|_{L^\infty (Q)}
\right ).
\ee
Since
$v_{\tau -t} = z(-t, \omega) u_{\tau -t}$
    with $u_{\tau -t} \in  D(\tau -t, \theta_{ -t} \omega )$,  
 we find that
  there exists $T=T(\tau, \omega, D) \ge  1$ such that
for all $t \ge T$,
\be\label{lcoms3_p3}
  c_1 e^{ -(\lambda -\beta) (r-\tau +t)} \| v_{\tau -t}\|_\cqzero
\le   1.
\ee
On the other hand, by \eqref{gh1} one can verify that the following
integral is convergent:
\be\label{lcoms3_p4}
\int^\tau_{ -\infty} e^{ (\lambda -\beta) s} z(s, \omega^{-\tau})
\left (
\| h(s, \cdot)\|_{L^\infty (Q)}
+ \| g(s, \cdot) \|_{L^\infty (Q)}
\right )ds <  \infty.
\ee
    It follows  from \eqref{lcoms3_p1}-\eqref{lcoms3_p4} that
    for all $t\ge T$,
     $$
     \| \vt(r, \tau -t, \omega^{-\tau}, g, h, \vt_{\tau -t} )  \|_\cqzero
     $$ 
 $$
\le 1 + 
c_1 e^{ \lambda -\beta } e^{ -(\lambda -\beta) \tau}
\int^\tau_{-\infty} e^{ (\lambda -\beta) s} z(s, \omega^{-\tau})
\left (
\| h(s, \cdot)\|_{L^\infty (Q)}
+ \| g(s, \cdot) \|_{L^\infty (Q)}
\right ),
$$
which along with 
the comparison principle yields the lemma.
\end{proof}

 Based on Lemma \ref{est1},   we have the following
 uniform estimates on solutions of problem \eqref{v1}-\eqref{v3}
 which imply the compactness of the solution operators.
 
    \begin{lem}
\label{est2}
Suppose   \eqref{f1}-\eqref{gh2}   hold
and $g \in L^\infty_{loc} (\R, L^\infty(Q))$. 
Then for every $\tau \in \R$, $\omega \in \Omega$  
 and $D=\{D(\tau, \omega)
: \tau \in \R,  \omega \in \Omega\}  \in \cald$,
 there exists  $T=T(\tau, \omega,  D) \ge 1$  
 such that for all $t \ge T$ and $\gamma \in [0, 1)$, 
 the solution
 $v$ of  problem  \eqref{v1}-\eqref{v3}    satisfies
 $$
 \| A_0^\gamma 
 v(\tau, \tau -t,  \omega^{-\tau}, f, g, v_{\tau -t})\|_\cqzero
 \le C,
 $$
 where $ v_{\tau -t} = z(-t, \omega) u_{\tau -t}$
    with $u_{\tau -t} \in  D(\tau -t, \theta_{ -t} \omega )$,
 $A_0 = -\Delta$ with homogeneous 
 Dirichlet boundary condition, 
 and
 $C$ is a positive number depending on $\tau$
 and $\omega$.
\end{lem}

\begin{proof}  
Note   that $z(s, \omega) =e^{-\alpha \omega (s)}$ and 
  $z^{-1}(s, \omega) =e^{\alpha \omega (s)}$
  are both continuous in $s \in \R$.
  Therefore, by 
  \eqref{f1}  and
  \eqref{est1_a1} we find that,
  for every $\tau \in \R$, 
  $\omega \in \Omega$, 
  $s \in [\tau -1, \tau]$
  and $t \ge T$, 
  \be\label{est1_p2}
 z(s, \omega^{-\tau} )
 |f \left ( s, \cdot , z^{-1} (s,\omega^{-\tau} ) \;
   v \left ( s,  \tau -t,  \omega^{-\tau}, f, g, v_{\tau -t}
 \right )
  \right  ) |
  \le c_1,
  \ee 
  where $c_1 =c_1(\tau, \omega)$ is a positive number.
 By \eqref{nonv} we get, for each
 $\tau \in \R$, $\omega \in \Omega$ and
 $t\ge1$,
 $$
 v(\tau, \tau -t,  \omega^{-\tau}, f, g, v_{\tau -t})
 = v \left ( \tau, \tau - 1,  \omega^{-\tau}, f, g,  
 v(\tau -1,  \tau -t,  \omega^{-\tau}, f, g, v_{\tau -t})
 \right )
 $$
 $$
 =e^{-A_0}   
 v(\tau -1,  \tau -t,  \omega^{-\tau}, f, g, v_{\tau -t})
 $$
 \be\label{est1_p3}
+\int_{\tau-1}^\tau z(s, \omega^{-\tau} )
e^{-A_0 (\tau -s) }
\left (
 f \left ( s, \cdot , z^{-1} (s,\omega^{-\tau} ) \;
   v \left ( s,  \tau -t,  \omega^{-\tau}, f, g, v_{\tau -t}
 \right )
  \right  ) 
  + g(s, \cdot)
\right ) ds.
\ee
It follows from 
\eqref{est1_a1}, 
\eqref{est1_p2}
and \eqref{est1_p3}
that there exists $T=T(\tau, \omega, D) \ge 1$ such that
 for each
 $\gamma \in [0, 1)$, 
 $\tau \in \R$, $\omega \in \Omega$ and
 $t \ge T$,
 $$
 \| A_0^\gamma 
 v(\tau, \tau -t,  \omega^{-\tau}, f, g, v_{\tau -t})\|_\cqzero
 $$
 $$
 \le 
 \| A_0^\gamma e^{-A_0}   
 v(\tau -1,  \tau -t,  \omega^{-\tau}, f, g, v_{\tau -t})
 \|_\cqzero
 $$
$$
+\int_{\tau-1}^\tau z(s, \omega^{-\tau} )
 \| A_0^\gamma e^{-A_0 (\tau -s) }
\left (
 f \left ( s, \cdot , z^{-1} (s,\omega^{-\tau} ) \;
   v \left ( s,  \tau -t,  \omega^{-\tau}, f, g, v_{\tau -t}
 \right )
  \right  ) 
  + g(s, \cdot)
\right ) \|_{L^\infty(Q)}  
$$
$$
 \le 
 c_2  e^{-\lambda} 
 \|v(\tau -1,  \tau -t,  \omega^{-\tau}, f, g, v_{\tau -t})
 \|_\cqzero
 $$
$$
+ c_2 \int_{\tau-1}^\tau 
e^{ -\lambda (\tau -s)}
{\frac {z(s, \omega^{-\tau} )}{ (\tau -s)^\gamma}}
\left (
 \| f \left ( s, \cdot , z^{-1} (s,\omega^{-\tau} ) \;
   v \left ( s,  \tau -t,  \omega^{-\tau}, f, g, v_{\tau -t}
 \right )
  \right  )  \|_\cqzero 
  +  \| g(s, \cdot)\|_{L^\infty (Q)} 
\right )   
$$
$$
\le c_3 +
c_3  (1 + 
\| g \|_{L^\infty ((\tau -1, \tau), L^\infty (Q) )}  )
 \int_{\tau-1}^\tau 
 (\tau -s)^{-\gamma} ds.
$$
This completes    the proof.
\end{proof}

Let $D(A_0^\gamma)$ be the domain of $A_0^\gamma$
with $\gamma >0$. Then we know that for each
$\gamma>0$,  the embedding 
$D(A_0^\gamma) \hookrightarrow \cqzero$
is compact. This along with Lemma \ref{est2}
 immediately implies    the pullback asymptotic
 compactness of solutions of problem \eqref{v1}-\eqref{v3}
 in $\cqzero$ as stated below.
 
   \begin{lem}
\label{est3}
Suppose   \eqref{f1}-\eqref{gh2}   hold
and $g \in L^\infty_{loc} (\R, L^\infty(Q))$. 
Then for every $\tau \in \R$, $\omega \in \Omega$, 
 $D=\{D(\tau, \omega)
: \tau \in \R,  \omega \in \Omega\}  \in \cald$ and 
$t_n \to \infty$,
 $v_{0,n} = z(-t_n, \omega) u_{0,n}$ with 
  $u_{0,n} \in D(\tau -t_n, \theta_{ -t_n} \omega )$,  the sequence
 $v(\tau, \tau -t_n,   \omega^{-\tau}, f, g,   v_{0,n}  ) $ 
 of solutions of problem \eqref{v1}-\eqref{v3}
   has a
   convergent
subsequence in   $\cqzero$.
\end{lem}

   W now prove the existence of
         closed measurable
 $\cald$-pullback absorbing sets   for problem \eqref{rde1}-\eqref{rde3}.
 
  \begin{lem}
 \label{est7}
  Suppose   \eqref{f1}-\eqref{gh2}   hold. 
 Then the continuous cocycle $\Phi$ associated with
 problem \eqref{rde1}-\eqref{rde3} has a closed measurable
 $\cald$-pullback absorbing set
 $K =\{ K(\tau, \omega): \tau \in \R,  \omega \in \Omega\}$
 $\in \cald$ which is given by
  $$
 K(\tau, \omega) = \{ u\in  \cqzero : \|u \|_\cqzero
 \le  R(\tau, \omega) \}  
$$
 with
$$
R(\tau, \omega) = 
M+  Mz^{-1} (\tau,  \theta_{ -\tau} \omega)
  \int_{-\infty}^\tau
    e^{(\lambda -\beta) (s-\tau)}  z(s,  \theta_{ -\tau}\omega )
    \left (\| h(s,\cdot)\|_{L^\infty (Q)} 
    +
    \| g(s,\cdot)\|_{L^\infty (Q)}  \right ) ds,
  $$
  where  $M$ is a  positive number  depending on
  $\lambda$ and $\beta$, but independent of 
  $\tau, \omega$  and $D$.
  \end{lem}

\begin{proof}
   By \eqref{vuid}  and    Lemma \ref{est1},  
  there  exists $T=T(\tau, \omega, D)  \ge 1$  such that   for   all
  $t \ge T$,     
  $$
  \| u (\tau, \tau -t,  \theta_{ -\tau} \omega, 
f, g, u_{\tau -t}  ) \|_\cqzero
= \|v (\tau, \tau -t,   \omega^{-\tau}, 
f, g,    z(-t,  \omega) u_{\tau -t}  ) \|_\cqzero
  $$
$$
\le M
 +
    M    \int_{-\infty}^\tau
    e^{(\lambda -\beta) (s-\tau) }  z(s,  \omega^{-\tau})
    \left (\| h(s,\cdot)\|_{L^\infty (Q)} 
    +
    \| g(s,\cdot)\|_{L^\infty (Q)}  \right ) ds
  $$
    $$
   \le
M +
    M z^{-1} (\tau, \theta_{ -\tau} \omega)   \int_{-\infty}^\tau
    e^{(\lambda -\beta) (s-\tau) }  z(s, \theta_{ -\tau}  \omega )
    \left (\| h(s,\cdot)\|_{L^\infty (Q)} 
    +
    \| g(s,\cdot)\|_{L^\infty (Q)}  \right ) ds,
    $$
    from which we get for all $t \ge T$, 
  $$
  \Phi (t, \tau-t, \theta_{ -t} \omega, D(\tau -t, \theta_{ -t} \omega ))
   \subseteq K(\tau, \omega).
  $$
   Note that, 
   for each $\tau \in \R$,
   $K(\tau, \cdot): \Omega \to 2^H$ is a  measurable set-valued mapping
because   $R(\tau, \cdot): \Omega \to \R$
is   $(\calf, \calb(\R))$-measurable.
Next, we  prove  $K$ is tempered which will complete
 the proof.
 Actually,
  for each $\tau \in \R$, $\omega \in \Omega$ and $r<0$,
     we have
   $$ \| K(\tau +  r, \theta_{ r} \omega) \|_\cqzero
   $$
   $$
  \le 
  M + M z^{-1}  (\tau +r, \theta_{-\tau} \omega )
    \int_{-\infty} ^{\tau +r}
   e^{ (\lambda -\beta) (s -\tau - r)} z(s, \theta_{ -\tau} \omega)
    \left (\| h(s,\cdot)\|_{L^\infty (Q)} 
    +
    \| g(s,\cdot)\|_{L^\infty (Q)}  \right ) ds
   $$
    $$
  \le 
  M + M  e^{\alpha \omega (r)}
    \int_{-\infty} ^{\tau +r}
   e^{ (\lambda -\beta) (s -\tau - r)} e^{-\alpha \omega (s-\tau) }
    \left (\| h(s,\cdot)\|_{L^\infty (Q)} 
    +
    \| g(s,\cdot)\|_{L^\infty (Q)}  \right ) ds
   $$
    \be\label{pest7_3}
  \le 
  M + M  e^{\alpha \omega (r)}
    \int_{-\infty} ^{0}
   e^{ (\lambda -\beta- \delta) s } e^{-\alpha \omega (s +r) }
   e^{\delta s}
    \left (\| h(s+\tau +r ,\cdot)\|_{L^\infty (Q)} 
    +
    \| g(s+\tau +r ,\cdot)\|_{L^\infty (Q)}  \right ) ds.
   \ee
   Given a positive number  $c$,   
  let  
  \be\label{pest7_3a1}
  \varepsilon 
  =\min \{ \lambda - \beta  -\delta,  \ {\frac 14} c\}.
  \ee
  By \eqref{aspomega} we see   that  there
  exists $N_1 <0$  such that  
  \be\label{pest7_4}
  |   \alpha  \ \omega (r) |
  \le -\varepsilon r \quad
  \mbox{ for all }  \ r \le  N_1.
  \ee
  By \eqref{pest7_3}-\eqref{pest7_4}   we have,
  for  all  $r \le N_1$, 
  $$ 
  \| K(\tau +  r, \theta_{ r} \omega) \|_\cqzero
   $$
   $$
  \le 
  M + M  e^{-2 \varepsilon  r}
    \int_{-\infty} ^{0}
   e^{ (\lambda -\beta- \delta -\varepsilon) s }  
   e^{\delta s}
    \left (\| h(s+\tau +r ,\cdot)\|_{L^\infty (Q)} 
    +
    \| g(s+\tau +r ,\cdot)\|_{L^\infty (Q)}  \right ) ds
   $$
    \be\label{pest7_8}
  \le 
  M + M  e^{- {\frac 12} cr }
    \int_{-\infty} ^{0} 
   e^{\delta s}
    \left (\| h(s+\tau +r ,\cdot)\|_{L^\infty (Q)} 
    +
    \| g(s+\tau +r ,\cdot)\|_{L^\infty (Q)}  \right ) ds.
   \ee
  By  \eqref{pest7_8}  and  \eqref{gh2}  we find,   
 for every positive constant $c$,
 $$
 \lim_{r \to -  \infty}
 e^{c  r} \| K(\tau +  r, \theta_{ r} \omega) \|_\cqzero
 =0,
 $$
 and hence 
 $K = \{ K(\tau, \omega): \tau \in \R, \omega \in \Omega \}$
 is tempered, as desired. 
  \end{proof}

  We now present    the $\cald$-pullback   asymptotic 
  compactness of  problem
  \eqref{rde1}-\eqref{rde3}.

 \begin{lem}
 \label{est8}
Suppose   \eqref{f1}-\eqref{gh2}   hold
and $g \in L^\infty_{loc} (\R, L^\infty(Q))$. 
Then the continuous cocycle $\Phi$ associated with
 problem \eqref{rde1}-\eqref{rde3}  
  is $\cald$-pullback
 asymptotically compact  in $\cqzero$,
that is, for  every $\tau \in \R$, $\omega \in \Omega$, 
 $D=\{D(\tau, \omega): \tau \in \R, \omega \in \Omega \}$
 $  \in \cald$,
and $t_n \to \infty$,
 $u_{0,n}  \in D(\tau -t_n, \theta_{ -t_n} \omega )$,  the sequence
 $\Phi(t_n, \tau -t_n,  \theta_{ -t_n} \omega,   u_{0,n}  ) $   has a
   convergent
subsequence in $\cqzero $.
\end{lem}

\begin{proof}
This lemma is an immediate consequence of
equality 
\eqref{vuid} and Lemma \ref{est3}.
\end{proof}

As mentioned before,   if $\Phi$ is 
$\cald$-pullback
 asymptotically compact,  then the 
   attractor   $\cala$ of $\Phi$
is $(\calf, \calb (X))-$measurable 
as  proved   in
\cite{wan8}.  The    measurability of $\cala$
  with respect to 
   the $P$-completion
 of $\calf$ was proved    in \cite{wan4}.
 The author is also referred to
 \cite{cra1, fla1, schm1}
 for measurability of random attractors.

  We  are now   ready to prove  
    the existence of 
   pullback attractors for
  problem \eqref{rde1}-\eqref{rde3}.  
  
  \begin{thm}
\label{thmrde1}
Suppose   \eqref{f1}-\eqref{gh2}   hold
and $g \in L^\infty_{loc} (\R, L^\infty(Q))$. 
Then the continuous cocycle $\Phi$ associated with
 problem \eqref{rde1}-\eqref{rde3}  
   has a unique $\cald$-pullback attractor $\cala
   =\{\cala(\tau, \omega):
      \tau \in \R, \ \omega \in \Omega \} \in \cald$
 in $\cqzero$.  Furthermore,
 for each $\tau  \in \R$   and
$\omega \in \Omega$,
$$
\cala (\tau, \omega)
=\Omega(K, \tau, \omega)
=\bigcup_{B \in \cald} \Omega(B, \tau, \omega)
$$
$$
 =\{\psi(0, \tau, \omega): \psi \mbox{ is a   }  \cald {\rm -}
 \mbox{complete solution  of } \Phi\} ,
$$
$$
 =\{\xi( \tau, \omega): \xi \mbox{ is a   }  \cald {\rm -}
 \mbox{complete  quasi-solution  of } \Phi\} ,
$$
where $K$ is the closed measurable $\cald$-pullback  absorbing set
of $\Phi$ given by Lemma \ref{est7}.

If,  in addition, 
   there exists $T>0$ such that
  for all  $     t \in \R, \ x \in Q$
   and   $s \in \R$,
  \be
  \label{perso1}
  f(t + T, x , s ) =  f(t,x,s) ,
  \ g(t+T,x ) = g(t,x)
  \ \mbox{ and } \  h(t+T,x ) = h(t,x),
  \ee
  then the attractor $\cala$ is $T$-periodic, 
  i.e., $\cala(\tau +T, \omega) =\cala(\tau, \omega)$.
\end{thm}

\begin{proof}
By Lemma \ref{est8} we know that
  $\Phi$ is   $\cald$-pullback 
asymptotically  compact in $\cqzero$.
Since $\Phi$ also has 
 a closed measurable $\cald$-pullback  absorbing set
 $K$,   it follows    from Proposition \ref{att}
 that  $\Phi$   has a unique
 $\cald$-pullback attractor 
 $\cala$ in $\cqzero$
 with the given structure.
 On the other hand, if \eqref{perso1} is fulfilled,
 then the cocycle $\Phi$ and the absorbing set $K$
 are both $T$-periodic, and  so is the attractor $\cala$.
 
\end{proof}

\section{Maximal and Minimal     Random  Complete    Solutions}
\setcounter{equation}{0}

In this section, we first  discuss
the  existence of tempered
complete quasi-solutions  
of problem \eqref{rde1}-\eqref{rde3}
 which are maximal and minimal
with respect to the random attractor
$\cala$. We then consider the existence of
 tempered random periodic  solutions.
  The stability of  these
  solutions is  also  examined.

 \begin{thm}
\label{max1}
Suppose   \eqref{f1}-\eqref{gh2}   hold
and $g \in L^\infty_{loc} (\R, L^\infty(Q))$. 
Then problem \eqref{rde1}-\eqref{rde3}  has
two tempered complete quasi-solutions $u^*$ and $u_*$ 
in $\cqzero$ such that  
 for all  $ \tau \in \R$ and $\omega \in \Omega$, 
  $u^* (\tau , \omega  )  \in  \cala (\tau, \omega)$,
    $u_*(\tau, \omega) \in  \cala (\tau, \omega)$
    and
\be
\label{max1_1}
  u_*(\tau, \omega )  (x)  \le u(x) \le u^*(\tau, \omega) (x)  \quad  \mbox{for   all } \
  u \in  \cala(\tau, \omega) \ \mbox{ and } \ 
x \in \rn,
\ee
where $ \cala$
is   the unique pullback attractor of
problem \eqref{rde1}-\eqref{rde3} in $\cqzero$.
The maximal complete quasi-solution
$u^*$ is asymptotically stable from above  in the sense that
for every  $ D=\{D(\tau, \omega) : \tau \in \R, \omega \in \Omega\}$ 
$ \in \cald$ 
and $\psi (\tau, \omega) \in D(\tau, \omega)$
with $\psi(\tau, \omega ) \ge u^*(\tau, \omega)$
for all $ \tau \in \R$ and $\omega \in \Omega$, 
the cocycle $\Phi$ associated with  
 problem \eqref{rde1}-\eqref{rde3} satisfies,
 for each  $ \tau \in \R$ and $\omega \in \Omega$, 
 \be
\label{max1_2}
\lim_{t \to \infty}
\Phi (t, \tau -t ,  \theta_{ -t} \omega, \psi (\tau -t,  \theta_{ -t} \omega ))
=  u^*(\tau, \omega), 
\quad \mbox{in } \ \cqzero.
\ee
The minimal  complete  quasi-solution
 $u_*$ is asymptotically stable from below  in the sense that
for every  $ D=\{D(\tau, \omega) : \tau \in \R, \omega \in \Omega\}$ 
$ \in \cald$ 
and $\psi (\tau, \omega) \in D(\tau, \omega)$
with $\psi(\tau, \omega ) \le  u_*(\tau, \omega)$
for all 
$ \tau \in \R$ and $\omega \in \Omega$, 
the cocycle $\Phi$ associated with  
 problem \eqref{rde1}-\eqref{rde3} satisfies,
 for each  $ \tau \in \R$ and $\omega \in \Omega$, 
 \be
\label{max1_3}
\lim_{t \to \infty}
\Phi (t, \tau -t ,  \theta_{ -t} \omega, \psi (\tau -t,  \theta_{ -t} \omega ))
=  u_*(\tau, \omega), 
\quad \mbox{in } \ \cqzero.
\ee
  \end{thm}

\begin{proof}
Let $\xi$ be the unique  tempered complete quasi-solution of
problem \eqref{lv1}-\eqref{lv3} in
$\cqzero$ as given by \eqref{lcoms2a1}.  
By
\eqref{f2} we  find  that $\xi$
and $-\xi$ are super- and sub-solutions of
problem \eqref{v1}-\eqref{v3}, respectively.
Next, we   prove   that
for every $ \tau  \in \R$
and $\omega \in \Omega$, there exist
$u^*(\tau, \omega)$ and $u_* (\tau, \omega)$ 
in $\cala (\tau, \omega)$
such that  
\be
\label{max1_p1}
\lim_{t \to  \infty}
\| u(\tau,  \tau -t, \theta_{ -\tau} \omega,
f, g, \xi (\tau -t, \theta_{ -t} \omega)) 
  - u^* (\tau, \omega) \|_\cqzero=0,
\ee
and
\be
\label{max1_p2}
 \lim_{t \to  \infty}
\| u(\tau,  \tau -t, \theta_{ -\tau} \omega,
f, g,  - \xi (\tau -t, \theta_{ -t} \omega)) 
  - u_* (\tau, \omega) \|_\cqzero=0,
\ee
where $u(\tau,  \tau -t, \theta_{ -\tau} \omega,
f, g,  \pm \xi (\tau -t, \theta_{ -t} \omega)) $
is the solution of problem \eqref{rde1}-\eqref{rde3}
with initial data $ \pm \xi (\tau -t, \theta_{ -t} \omega) $
 at initial time $\tau -t$.
 We will further  prove that
$u^*$ and $u_*$ are both  tempered complete 
quasi-solutions
and have the desired properties as stated in Theorem \ref{max1}.

By  \eqref{vuid}  we find that, 
for every $ \tau  \in \R$,
 $\omega \in \Omega$
 and $t \ge 0$, 
 \be
\label{max1_p3}
 u(\tau,  \tau -t, \theta_{ -\tau} \omega,
f, g, \xi (\tau -t, \theta_{ -t} \omega) )
=
 v(\tau,  \tau -t,  \omega^{-\tau},
f, g,  z(-t, \omega) \xi (\tau -t, \theta_{ -t} \omega) ).
\ee
By \eqref{lcoms2a1} we have
$$
 z(-t, \omega) \xi (\tau -t, \theta_{ -t} \omega)
 = \xi (\tau -t,   \omega^{-t}  ),
 $$
 which along with \eqref{max1_p3}  implies
 \be
\label{max1_p4}
 u(\tau,  \tau -t, \theta_{ -\tau} \omega,
f, g, \xi (\tau -t, \theta_{ -t} \omega) )
=
 v(\tau,  \tau -t,  \omega^{-\tau},
f, g,   \xi (\tau -t,   \omega^{-t}  )    ).
\ee
Let $t_1> t_2 \ge 0$. By the comparison principle we have
\be\label{max1_p5}
v(\tau -t_2,  \tau -t_1,  \omega^{-\tau},
f, g,   \xi (\tau -t_1,   \omega^{-t_1}  )    )
\le
\vt (\tau -t_2,  \tau -t_1,  \omega^{-\tau},
 g,  h,   \xi (\tau -t_1,   \omega^{-t_1}  )    ).
 \ee
 Since $\xi$ is a complete quasi-solution of 
 problem \eqref{lv1}-\eqref{lv3}, we get
 $$
 \vt (\tau -t_2,  \tau -t_1,  \omega^{-\tau},
 g,  h,   \xi (\tau -t_1,   \omega^{-t_1}  )    )
 =\xi (\tau -t_2, \omega^{-t_2}),
 $$
 which together with \eqref{max1_p5} shows that
 \be\label{max1_p6}
v(\tau -t_2,  \tau -t_1,  \omega^{-\tau},
f, g,   \xi (\tau -t_1,   \omega^{-t_1}  )    )
\le
\xi (\tau -t_2, \omega^{-t_2}).
 \ee
 By \eqref{max1_p6} and the 
 comparison principle, we obtain
$$
v(\tau ,  \tau -t_2,  \omega^{-\tau},
f, g,   \ 
v(\tau -t_2,  \tau -t_1,  \omega^{-\tau},
f, g,   \xi (\tau -t_1,   \omega^{-t_1}   )) )
$$
 $$
\le v(\tau ,  \tau -t_2,  \omega^{-\tau},
f, g,   \ 
\xi (\tau -t_2, \omega^{-t_2} )),
$$
which implies that  for all $t_1> t_2 \ge 0$, 
 \be\label{max1_P10}
v(\tau ,    \tau -t_1,  \omega^{-\tau},
f, g,   \xi (\tau -t_1,   \omega^{-t_1}   ) )
\le
v(\tau ,  \tau -t_2,  \omega^{-\tau},
f, g,   \ 
\xi (\tau -t_2, \omega^{-t_2} )).
\ee
Therefore, for each $\tau \in \R$  
and $\omega \in \Omega$,
$v(\tau ,    \tau -t,  \omega^{-\tau},
f, g,   \xi (\tau -t,   \omega^{-t}   ) )$
is monotone in $t \in \R^+$, and so is
$u (\tau ,    \tau -t,  \theta_{ -\tau} \omega,
f, g,   \xi (\tau -t,   \theta_{ -t} \omega  ) )$
by \eqref{max1_p4}.
Since $\xi$ is tempered,  
by the attracting property
of $\cala$, 
for each $\tau \in \R$   and
$\omega \in \Omega$,
\be\label{max1_p11}
\lim_{t \to \infty}
{\rm dist} _\cqzero
\left (
u (\tau ,    \tau -t,  \theta_{ -\tau} \omega,
f, g,   \xi (\tau -t,   \theta_{ -t} \omega  ) ),
  \ \cala (\tau, \omega)   \right )
=0.
\ee
By the compactness of  
$\cala (\tau, \omega)$ in 
$\cqzero$,  we find that
for  each $t \ge 0$, there is $u_0 (\tau, \omega, t)
 \in \cala (\tau, \omega)$   
such that
$$
\|u (\tau ,    \tau -t,  \theta_{ -\tau} \omega,
f, g,   \xi (\tau -t,   \theta_{ -t} \omega  ) )
- u_0 (\tau, \omega, t)
\|_\cqzero
$$
\be
\label{max1_p12}
=
{\rm dist} _\cqzero
\left (
u (\tau ,    \tau -t,  \theta_{ -\tau} \omega,
f, g,   \xi (\tau -t,   \theta_{ -t} \omega  ) ),
  \ \cala (\tau, \omega)   \right ).
\ee
By \eqref{max1_p11}-\eqref{max1_p12} we find
\be
\label{max1_p13}
\lim_{t \to \infty}
\|u (\tau ,    \tau -t,  \theta_{ -\tau} \omega,
f, g,   \xi (\tau -t,   \theta_{ -t} \omega  ) )
- u_0 (\tau, \omega, t)
\|_\cqzero = 0.
\ee
Since 
$u_0 (\tau, \omega, t) \in \cala (\tau, \omega)$
for all $t \in \R^+$ and 
$\cala (\tau, \omega)$
is compact in $\cqzero$, we  see that there 
exist $u^*(\tau, \omega) \in
\cala (\tau, \omega)$
and a sequence $\{t_n\}_{n=1}^\infty$ with
$t_n \to \infty$ such that
$u_0(\tau, \omega, t_n) \to u^*(\tau, \omega)$
in $\cqzero$. 
This and \eqref{max1_p13} imply
\be
\label{max1_p14}
\lim_{n \to \infty}
\|u (\tau ,    \tau -t_n,  \theta_{ -\tau} \omega,
f, g,   \xi (\tau -t_n,   \theta_{ -t_n} \omega  ) )
- u^* (\tau, \omega)
\|_\cqzero = 0.
\ee
Since
$u (\tau ,    \tau -t,  \theta_{ -\tau} \omega,
f, g,   \xi (\tau -t,   \theta_{ -t} \omega  ) )$
is monotone in $t \in \R^+$, we find from
\eqref{max1_p14}
that, for each $\tau \in \R$  and
$\omega \in \Omega$,
\be
\label{max1_p15}
\lim_{t \to \infty}
\|u (\tau ,    \tau -t,  \theta_{ -\tau} \omega,
f, g,   \xi (\tau -t,   \theta_{ -t} \omega  ) )
- u^* (\tau, \omega)
\|_\cqzero = 0.
\ee
Thus  \eqref{max1_p1} follows. 
Given $r \in \R$, replacing $\tau$ by $r +\tau$
and $\omega$ by $\theta_{ r} \omega$
in \eqref{max1_p15} we obtain
\be
\label{max1_p16}
\lim_{t \to \infty}
\|u (r+ \tau ,    r+ \tau -t,  \theta_{ -\tau} \omega,
f, g,   \xi (r+ \tau -t,   \theta_{ r-t} \omega  ) )
- u^* (r+ \tau, \theta_{ r} \omega)
\|_\cqzero = 0.
\ee

We  now prove that
$u^*$ is a complete quasi-solution of problem
\eqref{rde1}-\eqref{rde3}.
By \eqref{max1_p15} and the continuity
of solutions in initial data in $\cqzero$, we obtain that
for every $r \in \R^+$, 
$\tau \in \R$ and $\omega \in \Omega$,
$$
u(r + \tau, \tau, \theta_{ -\tau}, f, g, u^*(\tau, \omega))
$$
$$
=
\lim_{t \to \infty}
u(r + \tau, \tau, \theta_{ -\tau}, f, g,
 u (\tau ,    \tau -t,  \theta_{ -\tau} \omega,
f, g,   \xi (\tau -t,   \theta_{ -t} \omega  ) ))
$$
$$
=
\lim_{t \to \infty}
u(r + \tau,  \tau -t,  \theta_{ -\tau} \omega,
f, g,   \xi (\tau -t,   \theta_{ -t} \omega  ) )
$$
$$
=
\lim_{s \to \infty}
u(r + \tau,  r+ \tau -s,  \theta_{ -\tau} \omega,
f, g,   \xi (r+\tau -s,   \theta_{ -s } \theta_{ r }  \omega  ) )
\quad ( {\rm where } \ s  = r+t)
$$
\be\label{max1_p17}
=
\lim_{t \to \infty}
u(r + \tau,  r+ \tau -t,  \theta_{ -\tau} \omega,
f, g,   \xi (r+\tau -t,   \theta_{ -t } \theta_{ r }  \omega  ) ).
\ee
It follows from
\eqref{max1_p16}-\eqref{max1_p17} that
for every  $r \in \R^+$, 
$\tau \in \R$
and $\omega \in \Omega$,
 $$ u (r+ \tau ,    \tau  ,  \theta_{ -\tau} \omega,
f, g,   u^* (\tau,  \omega) )
= u^*(r+\tau,  \theta_{ r}  \omega ),
$$
which shows that for every  $r \in \R^+$, 
$ \tau \in \R$
and $\omega \in \Omega$,
\be\label{max1_p20}
\Phi (r, \tau, \omega, u^* (\tau,  \omega) )
= u^*(r+\tau,  \theta_{ r}  \omega ).
\ee
By definition we see that $u^*$ is a complete
quasi-solution of $\Phi$. Since
$\cala$ is tempered   and 
$u^* (\tau,  \omega) \in \cala (\tau, \omega)$
for all $\tau \in \R$   and
$\omega \in \Omega$, we know   that
$u^*$ is also tempered.

 By a similar argument, one can show that there
 exists   a complete  quasi-solution  
 $u_*$ such that for every  $\tau \in \R$
 and $\omega \in \Omega$,    
$u_*(\tau, \omega) \in  \cala(\tau, \omega)$  and
\eqref{max1_p2} is fulfilled.

We now prove \eqref{max1_1}.
Let $r, s \in \R^+$,  $\tau \in \R$
and  $\omega \in \Omega$.
Replacing $t$ by $r$,
 $\tau$ by $\tau -s$ and
$\omega$ by $\omega^{-s}$ in \eqref{lcoms2a2}, we get
$$
\|\vt (\tau -s, \tau-s-r, \omega^{-\tau}, g, h,
\vt_{\tau -s-r} )  - \xi (\tau -s, \omega^{-s})
\|_\cqzero 
$$
$$
\le c e^{-(\lambda -\beta) r} \left (
\| \vt_{\tau -s-r} \|_\cqzero
+ \| \xi (\tau -s -r,  \omega^{-s-r} )\|_\cqzero
\right ).
$$
Letting  $r = t - s$ with $t \ge s$ in the above, we obtain
$$
\|\vt (\tau -s, \tau-t, \omega^{-\tau}, g, h,
\vt_{\tau - t} )  - \xi (\tau -s, \omega^{-s})
\|_\cqzero 
$$
\be\label{max1_p30}
\le c e^{-(\lambda -\beta)  (t-s) } \left (
\| \vt_{\tau -t} \|_\cqzero
+ \| \xi (\tau - t,  \omega^{-t} )\|_\cqzero
\right ).
\ee
Suppose $D\in \cald$ and
$u_{\tau -t} \in D(\tau -t, \theta_{ -t} \omega)$.
Letting  $\vt_{\tau -t} = z(-t, \omega) |u_{\tau -t} |$
in \eqref{max1_p30}, we get
$$
\|\vt (\tau -s, \tau-t, \omega^{-\tau}, g, h,
 z(-t, \omega) |u_{\tau -t} | )  - \xi (\tau -s, \omega^{-s})
\|_\cqzero 
$$
\be\label{max1_p31}
\le c e^{-(\lambda -\beta) (t-s) } \left (
\|  z(-t, \omega) u_{\tau -t}   \|_\cqzero
+ \| \xi (\tau - t,  \omega^{-t} )\|_\cqzero
\right ).
\ee
Taking the limit of \eqref{max1_p31} as $t \to \infty$, we obtain,
for every $s \in \R^+$, $\tau \in \R$  and $\omega \in \Omega$,
\be\label{max1_p32}
\lim_{t \to \infty}
\|\vt (\tau -s, \tau-t, \omega^{-\tau}, g, h,
 z(-t, \omega) |u_{\tau -t} | )  - \xi (\tau -s, \omega^{-s})
\|_\cqzero  =0.
\ee
By \eqref{max1_p32} and  the 
continuity of solutions, we find that
\be\label{max1_p33}
\lim_{t \to \infty}
v(\tau, \tau -s, \omega^{-\tau}, f, g,
\vt (\tau -s, \tau-t, \omega^{-\tau}, g, h,
 z(-t, \omega) |u_{\tau -t} | ) )
 = v(\tau, \tau -s, \omega^{-\tau}, f, g, \xi (\tau -s, \omega^{-s} ))
\ee
in $\cqzero$. 
By  the  comparison  principle,  we have
$$
  v  (\tau -s, \tau-t, \omega^{-\tau}, f,  g, 
 z(-t, \omega) |u_{\tau -t} | )
 \le \vt (\tau -s, \tau-t, \omega^{-\tau}, g, h,
 z(-t, \omega) |u_{\tau -t} | ),
 $$
 which along with the comparison principle again implies
 $$
 v(\tau, \tau -s, \omega^{-\tau}, f, g,
v (\tau -s, \tau-t, \omega^{-\tau},  f, g, 
 z(-t, \omega) u_{\tau -t}   ))
 $$
 $$
 \le  v(\tau, \tau -s, \omega^{-\tau}, f, g, 
 \vt (\tau -s, \tau-t, \omega^{-\tau}, g, h,
 z(-t, \omega) |u_{\tau -t} |  )).
 $$
 In other words, we have
 \be\label{max1_p34}
 v (\tau , \tau-t, \omega^{-\tau},  f, g, 
 z(-t, \omega) u_{\tau -t}   )
 \le 
  v(\tau, \tau -s, \omega^{-\tau}, f, g, 
 \vt (\tau -s, \tau-t, \omega^{-\tau}, g, h,
 z(-t, \omega) |u_{\tau -t} |  )).
\ee
Letting $t \to \infty$ in \eqref{max1_p34}, by
\eqref{max1_p33} we get
\be\label{max1_p35}
\limsup_{t \to \infty}
 v (\tau , \tau-t, \omega^{-\tau},  f, g, 
 z(-t, \omega) u_{\tau -t}   )
 \le
  v(\tau, \tau -s, \omega^{-\tau}, f, g, 
 \xi(\tau -s, \omega^{-s} ) ).
 \ee
It  follows from 
   \eqref{vuid},  \eqref{max1_p4}  and \eqref{max1_p35}
   that  for all $ s \in \R^+$, 
\be\label{max1_p40}
\limsup_{t \to \infty}
 u  (\tau , \tau-t, \theta_{ -\tau}  \omega,  f, g, 
   u_{\tau -t}   )
 \le
  u (\tau, \tau -s, \theta_{ -\tau} \omega , f, g, 
 \xi(\tau -s, \theta_{ -s}  \omega  ) ).
 \ee
 Letting $s \to \infty$ in \eqref{max1_p40}, 
 by \eqref{max1_p15} we get,   for 
 every $\tau \in \R  $  and $\omega \in \Omega$,
 \be\label{max1_p41}
\limsup_{t \to \infty}
 u  (\tau , \tau-t, \theta_{ -\tau}  \omega,  f, g, 
   u_{\tau -t}   )
 \le
  u^* (\tau, \omega), 
  \quad \mbox{uniformly on } \ {\overline {Q}}.
 \ee
   Given $ \tau \in \R$,
    $\omega \in \Omega$  and $u\in  \cala (\tau, \omega)$, by
 the invariance of  $\cala$,  we find
 that, for every $t > 0$, there is $u_{\tau -t}
 \in \cala (\tau -t, \theta_{ -t} \omega )$
 such that
 $u = 
 u  (\tau , \tau-t, \theta_{ -\tau}  \omega,  f, g, 
   u_{\tau -t}   )$. Therefore, 
   by \eqref{max1_p41} we find 
\be
\label{max1_p42}
 u(x) \le u^*(\tau, \omega) (x) \quad
 \mbox{for all } \ x \in Q.
 \ee
 By an analogous argument, one can check that
$u(x) \ge  u_*(\tau, \omega) (x)$    for all  $  x \in Q$,
and thus \eqref{max1_1} follows.

We now consider the stability of $u^*$  and $u_*$.
    Suppose $ D  \in \cald$ 
and $\psi (\tau, \omega ) \in D(\tau, \omega)$
with $ \psi (\tau, \omega ) \ge  u^* (\tau, \omega )$
for all $\tau \in \R$  and $\omega \in \Omega$.
   By the comparison principle we get,
   for  every $t \in \R^+$, $\tau \in \R$
   and $\omega \in \Omega$, 
$$
u(\tau, \tau -t, \theta_{ -\tau} \omega, f, g, 
\psi (\tau -t,  \theta_{ -t} \omega ))
\ge
u(\tau, \tau -t, \theta_{ -\tau} \omega, f, g, 
u^* (\tau -t,  \theta_{ -t} \omega ))
\ge u^*(\tau, \omega),
$$ 
which implies
\be
\label{max1_p50}
\liminf_{t \to \infty}
\Phi (t, \tau -t ,  \theta_{ -t} \omega, \psi (\tau -t,  \theta_{ -t} \omega ))
\ge u^*(\tau, \omega).
\ee
By \eqref{max1_p41} and \eqref{max1_p50} we find that
$$
\lim_{t \to \infty}
\Phi (t, \tau -t ,  \theta_{ -t} \omega, \psi (\tau -t,  \theta_{ -t} \omega ))
=  u^*(\tau, \omega), 
\quad \mbox{in } \ \cqzero,
$$
which gives \eqref{max1_2}.
The convergence of \eqref{max1_3}
can be proved similarly and the details are omitted.
  \end{proof}

  For random periodic solutions,  we have the following result.

   \begin{thm}
      \label{max2}
Let   \eqref{f1}-\eqref{f2}  hold.
Suppose  there exists $T>0$ 
such that  \eqref{perso1} is     valid. 
 Then the stochastic problem \eqref{rde1}-\eqref{rde3}
has two tempered random periodic solutions
$u^*$ and $u_*$ which satisfy \eqref{max1_1}-\eqref{max1_3}.
\end{thm}

\begin{proof} 
Let $u^*$ and $u_*$ be the maximal and minimal complete quasi-solutions
of problem \eqref{rde1}-\eqref{rde3} obtained in Theorem \ref{max1},
respectively. 
By \eqref{perso1} and Lemma \ref{lcoms2} we know that
the unique complete quasi-solution $\xi$ of problem \eqref{lv1}-\eqref{lv3}
is   periodic with period $T$.
Then it follows from  \eqref{max1_p15}
 that for each $\tau \in \R$  and $\omega \in \Omega$,
 $$
  u^* (\tau +T, \omega)=
\lim_{t \to \infty}
u (\tau+T ,    \tau+T -t,  \theta_{ -\tau -T} \omega,
f, g,   \xi (\tau+T -t,   \theta_{ -t} \omega  ) )
 $$
 $$
 =
 \lim_{t \to \infty}
u (\tau ,    \tau -t,  \theta_{ -\tau } \omega,
f, g,   \xi (\tau -t,   \theta_{ -t} \omega  ) )
=u^*(\tau, \omega).
 $$
 This shows   that $u^*$ is $T$-periodic.
 The T-periodicity of $u_*$ can be justified by a 
 similar argument, and the details are omitted.
\end{proof}

 By Theorem  \ref{max2} we find that
   the random attractor of problem
 \eqref{rde1}-\eqref{rde3} is  
 either trivial or it
 contains at least two different   random periodic solutions.
 Based on this  observation,  we can  
 prove the existence of multiple random periodic
 solutions  when the attractor
 of  the equation  is nontrivial. 
 This idea is demonstrated by the 
 Chafee-Infante  equation
 presented in the next section.

\section {Bifurcation of Random  Complete and Periodic Solutions}
\setcounter{equation}{0}

In this section, 
we apply the results of the previous section
to a specific model called  the Chafee-Infante  equation,
and prove the multiplicity of random complete and  periodic solutions.
As we will see later,   these solutions  undergo  a pitchfork
bifurcation when a parameter varies. 
 
The one-dimensional   autonomous  Chafee-Infante  equation 
is defined in $Q = (0, \pi)$:
\be
\label{cf1}
{\frac {\partial u}{\partial t}} -{\frac {\partial ^2 u}{\partial x^2}}
=\nu u -  \gamma_0 u^3, \quad x \in (0, \pi), \quad t>0,
\ee
with   boundary condition
\be
\label{cf2}
u(t, 0) = u(t, \pi) =0, \quad  t>0,\ee
and   initial condition
\be
\label{cf3}
u(0, x) =u_0(x), \quad x \in (0, \pi),\ee
where $\nu$  and $\gamma_0$
 are    positive constants.
 
The dynamics of problem \eqref{cf1}-\eqref{cf3}
is well understood   in the literature, see, e.g.,
\cite{cha1, hen1}. 
Let $A_0  = -\partial_{xx}$ with 
boundary condition \eqref{cf2}.
Then  the eigenvalues of $A_0$ are  
$ \lambda_n = n^2$ where $n$ is any positive  integer.
For every $n \in \N$, $A_0$ has an eigenvector
$e_n = \sin n x$ corresponding to $\lambda_n$.
It is known that $\{e_n\}_{n=1}^\infty$ is an orthogonal
basis of $\ltwo$. 
   For every $\nu \in (n^2, (n+1)^2)$
 with $n \in \N$,  it was proved by  Chafee and Infante in \cite{cha1} 
that  problem \eqref{cf1}-\eqref{cf3}
has exactly $2n+1$ equilibrium solutions.
Moreover,   using   a Liapunov function, one can show 
that problem \eqref{cf1}-\eqref{cf3}
has an $n$-dimensional global attractor 
in $\cqzero$ which is given by 
 the union of unstable manifolds
of the $(2n+1)$
 equilibrium solutions  for 
  $\nu \in (n^2, (n+1)^2)$
  (see, \cite{hen1}). 

Suppose  $\gamma: \R \to \R$ is  a bounded continuous function
and 
there exists  a positive number $\gamma_0$
such  that
\be
\label{ga1}
\gamma (t) \ge \gamma_0, \quad \mbox{  for  all  } \
t \in \R.
\ee 
 We now consider the stochastically perturbed 
 Chafee-Infante  equation,
 for  each  $\tau \in \R$,
 \be
\label{scf1}
{\frac {\partial u}{\partial t}} -{\frac {\partial ^2 u}{\partial x^2}}
=\nu u -  \gamma (t) u^3  
+  \alpha u \circ  {\frac { d\omega}{dt}},  \quad x \in (0, \pi), \quad t>\tau,
\ee
with   boundary condition
\be
\label{scf2}
u(t, 0) = u(t, \pi) =0, \quad  t> \tau,
\ee
and   initial condition
\be
\label{scf3}
u(\tau, x) =u_\tau (x), \quad x \in (0, \pi),
\ee
where $\alpha$ is a positive number,   and 
$\omega$  is 
the   two-sided real-valued Wiener 
process  described before.

The bifurcation and the structures of attractors 
of equation  \eqref{scf1}
have been investigated in \cite{carva1, carva2, lan1}
for $\alpha =0$ and in \cite{car8}
for constant $\gamma$.  We here  consider
the bifurcation of random complete  solutions
of problem \eqref{scf1}-\eqref{scf3}
when $\gamma$  depends on $t$. 
 As a by-product, we obtain the bifurcation of random periodic
 solutions  for periodic  $\gamma$.
Note   that $u=0$ is a solution
of problem \eqref{scf1}-\eqref{scf3}, 
and hence  the origin is a trivial
random complete  solution for every $\nu \in \R^+$.
Let $f(t,x, s) = \nu s -\gamma (t) s^3$
for all $t \in \R$, $x \in (0,\pi)$
and $s \in \R$.
Recall that $\lambda_1 =1$ is the first eigenvalue of
$A_0$ with \eqref{cf2}. By \eqref{ga1}, one can check   that
conditions \eqref{f1}-\eqref{f2}
are fulfilled  for any $\beta \in (0, \lambda_1)$ with an appropriate
positive constant function $h$.
Then it follows   from
Theorem  \ref{thmrde1} 
that the stochastic problem
\eqref{scf1}-\eqref{scf3} has a unique tempered  
pullback attractor
$\cala =\{ \cala (\tau, \omega): \tau \in \R,
\omega \in \Omega \}$.
It is evident   that
$0 \in \cala(\tau, \omega)$
for all $\tau \in \R$   and $\omega \in \omega$.
On the other hand, 
 by Theorem \ref{max1}  
we know   that 
problem \eqref{scf1}-\eqref{scf3} has two
tempered random complete  quasi-solutions
$u^*$ and $u_*$ in $\cqzero$ with properties
\eqref{max1_1}-\eqref{max1_3} and
$u^*(\tau, \omega)$,
$u_*(\tau, \omega) \in \cala(\tau, \omega)$
for all $\tau \in \R$
and $\omega \in \Omega$. 
This shows   that
$u^*$   and $u_*$ are the maximal and minimal
random complete quasi-solutions, respectively.
In addition, 
$u^*$   is  stable from above 
and $u_*$ is stable
  from below. 
  
  Note that for each 
  $t \in \R^+$,
  $\tau \in \R$
  and $\omega \in \Omega$, the solution 
  $u$ of problem \eqref{scf1}-\eqref{scf3}
  satisfies, 
  \be\label{bif1}
   u(\tau,  \tau -t, \theta_{ -\tau} \omega,
  - u_{\tau - t}  )
  =
  -   u(\tau,  \tau -t, \theta_{ -\tau} \omega,
   u_{\tau - t}  ).
   \ee
   By \eqref{max1_p1}-\eqref{max1_p2}
   and \eqref{bif1}   we get,
   for each $\tau \in \R$  and $\omega \in \Omega$,
   $
   u_* (\tau, \omega)
   =- u^*(\tau, \omega)$.
   From  this   and \eqref{max1_1}   we obtain,
   for each $\tau \in \R$  and $\omega \in \Omega$,
   \be\label{bif2}
   -u^*(\tau, \omega) (x)
   \le u(x)
   \le u^*(\tau, \omega)
   \quad
   \mbox{ for  all  } \ u \in \cala(\tau, \omega)
   \  \mbox {   and } \ x \in  \R^n.
   \ee
   
   If $\nu \in (0, \lambda_1)$,  then it  is clear that
   all solutions of problem \eqref{scf1}-\eqref{scf3}
   converge to zero. Therefore, the attractor $\cala$
   is trivial and $u^* =u_* =0$. This shows   that
   zero is the only random 
   complete quasi-solution of 
   problem \eqref{scf1}-\eqref{scf3}
   in this case.
   
   If $\nu >\lambda_1$,  then  the origin becomes
      unstable and hence $\cala$ is nontrivial.
      This along with \eqref{bif2} implies that
      $u^* \neq 0$. So, in this case,
      problem \eqref{scf1}-\eqref{scf3} has
      three different random  complete quasi-solutions:
      $u=0$,  $u=u^*  $  and $u=u_*= -u^* $.
      We will show   that
      $\nu =\lambda_1$ is  actually
      a bifurcation point.
       For that purpose, 
      we  need to prove 
   $\cala$ is   trivial
   when $\nu =\lambda_1$. 
   
   \begin{lem}
   \label{lbif1}
   Suppose $\gamma$ is a  bounded continuous  
   function  which satisfies
   \eqref{ga1}.
   If $\nu =\lambda_1 =1$,  then 
   $u=0$ is the unique tempered 
   random complete quasi-solution of 
     problem \eqref{scf1}-\eqref{scf3}.
    In this case, the random attractor
    $\cala$ is trivial.
     \end{lem}
     
     \begin{proof}
     By \eqref{bif2}
       we only need to show
       $u^* =0$.
       Given $\tau \in \R$ and $\omega \in \Omega$,
       for each $t \ge \tau$  we denote by
       $u(t, \tau, \omega, u^*(\tau, \theta_{ \tau} \omega))$
       the solution of
       problem \eqref{scf1}-\eqref{scf3}
       with initial condition
       $u^*(\tau, \theta_{ \tau} \omega)$
       at initial time $\tau$.
       Since $u^*$ is a complete quasi-solution, we find  
       that  for every  $\tau \in \R$, $\omega \in \Omega$
       and  $t \ge \tau$,
       \be\label{lbif_p1}
       u(t, \tau, \omega, u^*(\tau, \theta_{ \tau} \omega))
       = u^* (t, \theta_{ t} \omega)
       \ge  0 .
       \ee
       Since $\{\sin n x \}_{n=1}^\infty$
       is an orthogonal basis of $\ltwo$, we may write
       \be\label{lbif_p2}
       u(t, \tau, \omega, u^*(\tau, \theta_{ \tau} \omega))
       = \sum_{n=1}^\infty
       a_n (t, \tau, \omega,  a_{n, \tau} ) \sin n x
       =u_1 + u_2 \quad \mbox{in } \ \ltwo,
       \ee
       where
       $u_1 = a_1 (t, \tau, \omega,   a_{1, \tau})  \sin x$ 
        and $u_2 = \sum\limits_{n=2}^\infty
       a_n (t, \tau, \omega,   a_{n, \tau} ) \sin n x$. 
       By \eqref{lbif_p1}-\eqref{lbif_p2}  we have
       \be\label{ibif_p3}
       a_1 (t, \tau, \omega,   a_{1, \tau}) = {\frac 2\pi}\int_0^\pi
         u(t, \tau, \omega, u^*(\tau, \theta_{ \tau} \omega))
          \sin x  \ dx  \ \ge 0.
       \ee
       By  \eqref{scf1}   with
       $\nu = \lambda_1 =1$   we get   
       \be
       \label{ibif_p5}
       {\frac {d a_1}{dt}}
       = - {\frac 2\pi} 
       \gamma (t)
       \int_0^\pi
       u^3 \sin x \  dx
       + \alpha a_1 \circ {\frac {d \omega}{dt}}.
       \ee
       By Holder\rq{} inequality we have
       $$
      \left ( \int_0^\pi u \sin x  \   dx
      \right  )^3
      \le 4 \int_0^\pi u^3  \sin x  \   dx,
      $$
      from which, \eqref{ga1} and
        \eqref{ibif_p3} we  get
       \be\label{ibif_p8}
       {\frac 2\pi} \gamma (t)
       \int_0^\pi
       u^3 \sin x  \ dx
       \ge {\frac {\pi^2}{16}}   \gamma_0 a_1^3.
       \ee
       By \eqref{ibif_p5}  and \eqref{ibif_p8} we get
         \be
       \label{ibif_p9}
       {\frac {d a_1}{dt}}
       \le  -  {\frac {\pi^2}{16}} \gamma_0 a_1^3
       + \alpha a_1 \circ {\frac {d \omega}{dt}}.
       \ee
       Solving  for $a_1$ from \eqref{ibif_p9}  we obtain,   for every
       $\tau \in \R$, $\omega \in \Omega$ and
       $t \ge \tau$,
       \be\label{ibif_p10}
       a_1(t, \tau, \omega,  a_{1, \tau})
      \le {\frac {
      e^{\alpha (\omega (t) -\omega (\tau))}  a_{1, \tau} }
      { \sqrt{1+ {\frac {\pi^2}{8}}\gamma_0  a_{1, \tau} ^2
      \int_\tau^t  e^{2\alpha (\omega(s) - \omega (\tau))} ds      }}}.
      \ee
      It follows from \eqref{ibif_p10} that,  for every
      $t \in \R^+$, $\tau \in \R$  and 
      $\omega \in \Omega$,
      \be\label{ibif_p11}
       a_1(\tau, \tau -t,  \theta_{ -\tau} \omega, a_{1, \tau -t})
       = a_1 (0, -t, \omega, a_{1, \tau -t} )
      \le {\frac {
      e^{-\alpha \omega (-t)  } a_{1, \tau -t}}
      { \sqrt{1+ {\frac {\pi^2}{8}}\gamma_0 a_{1, \tau-t} ^2
      \int_{-t} ^0  e^{2\alpha (\omega(s) - \omega (-t))} ds      }}}.
      \ee
      By \eqref{ibif_p11}  we find that
       for every $\tau \in \R$  and 
      $\omega \in \Omega$,
      \be\label{ibif_p20}
      \limsup_{t \to \infty}
       a_1(\tau, \tau -t,  \theta_{ -\tau} \omega, a_{1, \tau -t})
       \le 
       {\frac 2\pi} \left ({\frac 2{\gamma_0}}
       \right )^{\frac 12}
       \limsup_{t \to \infty}
      \left (
       \int_{-t} ^0  e^{2\alpha \omega(s) } ds 
       \right )^{-{\frac 12}}=0,
       \ee
       where we have used Lemma 2.3.41 in \cite{arn1}
       for the last limit.
       Since $u^*$ is a complete quasi-solution, by \eqref{ibif_p3}
       we obtain,
        \be\label{ibif_p21}
       a_1(\tau, \tau -t,  \theta_{ -\tau} \omega, a_{1, \tau -t})
       = {\frac 2\pi}\int_0^\pi
         u(\tau, \tau -t,  \theta_{ -\tau} \omega,
          u^*(\tau -t, \theta_{ -t} \omega))
          \sin x  \ dx   
          ={\frac 2\pi}\int_0^\pi 
          u^*(\tau ,   \omega )
          \sin x  \ dx
       \ee
       By \eqref{ibif_p20}-\eqref{ibif_p21}  we find   that
       for every $\tau \in \R$  and
       $\omega \in \Omega$,
       \be\label{ibif_p24}
       \int_0^\pi 
          u^*(\tau ,   \omega )
          \sin x  \ dx =0.
          \ee
          By \eqref{lbif_p1}, \eqref{ibif_p3}
          and \eqref{ibif_p24}  we get, 
          for every 
     $\tau \in \R$,   
       $\omega \in \Omega$
       and $t \ge \tau$,
        \be\label{ibif_p30}
       a_1 (t, \tau, \omega,  a_{1, \tau}) = {\frac 2\pi}\int_0^\pi
         u(t, \tau, \omega, u^*(\tau, \theta_{ \tau} \omega))
          \sin x  \ dx  
          =
          {\frac 2\pi}\int_0^\pi
          u^*(t, \theta_{ t} \omega )
          \sin x  \ dx  =0.
       \ee
       On the other hand, by 
       taking  the inner product of \eqref{scf1} 
       with $u_2$   in $\ltwo$,  we obtain,
       for every $\tau \in \R$, $\omega \in \Omega$
       and $t \ge \tau$, 
       $$
{\frac d{dt}} \| u_2 \|^2_\ltwo
       + 2 (\lambda_2 - \nu) \|u_2 \|^2 _\ltwo
       =-2\gamma (t) \int_0^\pi u^3 u_2 dx
  + 2\alpha \|u_2\|^2 _\ltwo \circ {\frac {d\omega}{dt}}
      $$
      \be
       \label{ibif_p40}
      =
      -2\gamma (t) \int_0^\pi u^4 dx
      + 2 \gamma  (t) \int_0^\pi u^3 u_1 dx
  + 2\alpha \|u_2\|^2 _\ltwo \circ {\frac {d\omega}{dt}}
  \le 
    2\alpha \|u_2\|^2 _\ltwo \circ {\frac {d\omega}{dt}},
    \ee
    where the last inequality  follows   from
    \eqref{ibif_p30}.
    Note that $\lambda_2 =4$   and $\nu =1$ in the present case.
    Therefore, it follows   from \eqref{ibif_p40}
    that
    for every $\tau \in \R$, $\omega \in \Omega$
       and $t \ge \tau$, 
       \be\label{ibif_p43}
       \| u_2 (t, \tau, \omega, u_{2,\tau})\|^2_\ltwo
       \le
       e^{6(\tau -t)} e^{ 2\alpha (\omega(\tau) - \omega (t))  }
       \|u_\tau \|^2_\ltwo.
       \ee
       By \eqref{ibif_p43}  we have,
    for every $ t \in \R^+$,
     $\tau \in \R$   and
     $\omega \in \Omega$,
       \be\label{ibif_p44}
       \| u_2 (\tau, \tau -t ,  \theta_{ -\tau} \omega, u^*_{2}
         (\tau -t, \theta_{-t} \omega )
         )\|^2_\ltwo
       \le
       e^{-6 t  } 
e^{-2\alpha  \omega (\tau)} 
       e^{ 2\alpha  \omega(\tau -t   ) }
       \| u^*  (\tau -t, \theta_{-t} \omega ) \|^2_\ltwo.
       \ee
       By \eqref{aspomega} and the temperedness of $u^*$
       in $\cqzero$, we obtain
       from \eqref{ibif_p44}
       that, 
        for every  
     $\tau \in \R$   and
     $\omega \in \Omega$,
     \be\label{ibif_p50}
      \lim_{t \to \infty}
       \| u_2 (\tau, \tau -t ,  \theta_{ -\tau} \omega, u^*_{2}
         (\tau -t, \theta_{-t} \omega )
         )\|^2_\ltwo =0.
         \ee
         By \eqref{lbif_p2}, \eqref{ibif_p30}
         and \eqref{ibif_p50} we find that,
         for every  
     $\tau \in \R$   and
     $\omega \in \Omega$,
        \be\label{ibif_p51}
      \lim_{t \to \infty}
       \| u  (\tau, \tau -t ,  \theta_{ -\tau} \omega, u^* 
         (\tau -t, \theta_{-t} \omega )
         )\|^2_\ltwo =0.
         \ee
         Since $u^*$ is a complete quasi-solution, from \eqref{ibif_p51}
         we get
         $u^*(\tau, \omega) = 0$ in $\ltwo$   for all $\tau\in \R$
         and $\omega \in \Omega$.
         This implies $u^*(\tau, \omega) (x) = 0$
         for all $x \in Q$  since $u^*(\tau, \omega)\in \cqzero$.
     \end{proof}
     
     We are now ready to prove the following bifurcation result
     for problem \eqref{scf1}-\eqref{scf3}.
     
   \begin{thm}
   \label{thmbi}
   Suppose $\gamma$ is a  bounded continuous  
   function satisfying 
   \eqref{ga1}.
   Then  the tempered 
   random  complete
     quasi-solutions of 
     problem \eqref{scf1}-\eqref{scf3}
    undergo a pitchfork bifurcation when
    the parameter $\nu$ crosses $\nu=1$
    from below.
    More precisely, if $\nu \le 1$, 
    problem \eqref{scf1}-\eqref{scf3} 
    has a unique tempered random
     complete
     quasi-solution
    $u=0$; if $\nu >1$,  the problem has 
    three  different tempered random  complete
     quasi-solutions:
    $u^*_\nu$, $-u^*_\nu$ and $0$.
    Furthermore, $u^*_\nu(\tau, \omega) \to 0$
    when $\nu \to 1$ for every $\tau\in\R$
    and $\omega \in \Omega$.
     \end{thm}
     
     \begin{proof}
      By Lemma \ref{lbif1}  and what we discussed before,
      if   $\nu \le 1$,  then
        $u=0$  is the only tempered random 
         complete
     quasi-solution
        of the stochastic problem \eqref{scf1}-\eqref{scf3} 
        and the  random attractor $\cala_\nu$ is trivial,
        i.e., $\cala_\nu(\tau, \omega) = \{0\}$
        for all $\tau \in \R$  and $\omega \in \Omega$.
        On  the other hand, if $\nu >1$ the problem has at least
        three tempered random  complete
     quasi-solutions:
        $u^*_\nu$, $-u^*_\nu$ and $0$.
        The attractor $\cala_\nu$ is nontrivial in this case.
        By  \cite{wan8},  we know 
         $\cala_\nu$ is upper-semicontinuous when
        $\nu \to 1$. This means   that 
        for every $\tau\in\R$
    and $\omega \in \Omega$,
        $ 
        \lim\limits_{\nu \to 1}
        {\rm dist}_{\cqzero} (\cala_\nu (\tau, \omega), \cala_1 (\tau, \omega)) =0$. 
        Since $u^*_\nu(\tau, \omega) \in \cala_\nu (\tau, \omega)$
        and  $\cala_1 (\tau, \omega) =\{0\}$,  we must
        have $u^*_\nu (\tau, \omega) \to 0$ as $\nu \to 1$.
      \end{proof}
    
    Suppose now 
   $\gamma : \R \to \R$ is $T$-periodic.  By Theorem \ref{max2} we find that
   the quasi-solution $u^*_\nu$ above is also $T$-periodic. As an immediate consequence,
   we obtain the following bifurcation of random periodic solutions.
   
   \begin{cor}
   \label{thmbiper}
   Suppose $\gamma : \R \to \R$ is a continuous periodic
   function  satisfying 
   \eqref{ga1}.
   Then  the tempered 
   random periodic
     solutions of 
     problem \eqref{scf1}-\eqref{scf3}
    undergo a pitchfork bifurcation when
    the parameter $\nu$ crosses $\nu=1$
    from below.
    More precisely, if $\nu \le 1$, 
    problem \eqref{scf1}-\eqref{scf3} 
    has a unique tempered random periodic solution
    $u=0$; if $\nu >1$,  the problem has 
    three  different tempered random periodic solutions:
    $u^*_\nu$, $-u^*_\nu$ and $0$.
    Furthermore, $u^*_\nu(\tau, \omega) \to 0$
    when $\nu \to 1$ for every $\tau\in\R$
    and $\omega \in \Omega$.
     \end{cor}

\end{document}